\documentclass[a4paper]{amsart}

\usepackage[utf8]{inputenc}
\usepackage[english]{babel}

\usepackage{amsmath} 
\usepackage{amssymb} 
\usepackage{amsfonts} 
\usepackage{amsthm} 
\usepackage{cancel} 

\usepackage{enumerate} 
\usepackage{enumitem} 
\usepackage{tikz-cd} 

\usepackage{hyperref}

\theoremstyle{plain}
	\newtheorem{theorem}{Theorem}[section]
	\newtheorem{proposition}[theorem]{Proposition}
	\newtheorem{lemma}[theorem]{Lemma}
	\newtheorem{corollary}[theorem]{Corollary}
\theoremstyle{definition}
	\newtheorem{definition}[theorem]{Definition}
\theoremstyle{remark}
	\newtheorem{remark}[theorem]{Remark}
	\newtheorem{example}[theorem]{Example}

\newcommand{\C}{\mathbb{C}}
\newcommand{\PP}{\mathbb{P}}

\title[Weddle loci and rank of partially symmetric tensors]{Weddle loci of linear systems of quadrics and the rank of partially symmetric tensors}

\author{Luca Chiantini}
\address[L. Chiantini]{Dipartimento di Ingegneria dell'Informazione e Scienze Matematiche\\Università degli Studi di Siena\\Via Roma 56\\53100 Siena, Italy}
\email{luca.chiantini@unisi.it}

\author{Filippo Fagioli}
\address[F. Fagioli]{Dipartimento di Matematica\\Università degli Studi di Roma Tor Vergata\\Via della Ricerca Scientifica 1\\00133 Roma, Italy}
\email{fagioli@mat.uniroma2.it}

\thanks{The authors are supported by the project PRIN2022 ``0-Dimensional Schemes, Tensor Theory, and Applications'' (code 2022E2Z4AK) and are members of GNSAGA (Gruppo Nazionale per le Strutture Algebriche, Geometriche e le loro Applicazioni) of the Istituto Nazionale di Alta Matematica (INdAM). The second author acknowledges the MUR Excellence Department Project MatMod@TOV awarded to the Department of Mathematics, University of Rome Tor Vergata, CUP E83C23000330006}
\date{\today}

\begin{document}

\maketitle

\begin{abstract}
	We establish a connection between properties of partially symmetric tensors (i.e. tensors associated to linear systems of quadric hypersurfaces) and the geometry of some related loci, generalization of the Weddle loci introduced in \cite{CFF+22} for their role in the study of configurations of points and interpolation problems.
	In particular, we consider linear systems of plane conics and linear systems of quadric surfaces, and show that when the associated tensors have low rank, then the singularities of the corresponding Weddle loci satisfy a (sharp) lower bound.
	Thus, we obtain a criterion to exclude that the rank of some partially symmetric tensors is too low.
	
	In the final section, devoted to partially symmetric $n\times n\times n$ tensors which lie in one component $M$ of a standard decomposition of the space of $3$-dimensional tensors (\cite{IR22}), we prove that the number of singular points of the Weddle locus associated to a general tensor in $M$ equals the (recursively defined) $n$-th Jacobsthal number.
\end{abstract}

\section*{Introduction}
	Through the paper we establish a link between multilinear properties of partially symmetric tensors (mainly $3$-dimensional tensors) and geometric properties of linear systems of hypersurfaces (mainly linear systems of quadrics), related with interpolation problems.  The link is based on some geometric properties of the Weddle locus of a linear system. 
	
	The classical notion of Weddle locus concerns linear systems $\mathcal L$ of quadric surfaces in $\PP^3$ passing through $6$ points in general position. The system contains a $2$-dimensional family of singular quadrics; in note of the classical paper \cite{Wed1850} Thomas Weddle correctly stated that the singular points of quadrics in $\mathcal L$ form a surface (now called the {\it Weddle surface} of $\mathcal L$). Many properties of the Weddle surfaces have been studied by A. Emch (see \cite{Emch25}), who also introduced their generalization to loci related with linear systems of quadric hypersurfaces in any $\PP^n$. 
	
	Recently, Weddle loci have been formalized and studied for their relations with special properties of configurations of points in projective spaces, and their projections (see \cite{CFF+22}), which are linked to interpolation problems (see \cite{CM21}). Since hypersurfaces of degree $d$  in $\PP^N$ are naturally associated to $d$-dimensional symmetric tensors, then linear systems are easily associated to $(d+1)$-dimensional, partially symmetric, tensors, by glueing together the tensors associated to any set of generators. In particular, linear systems of quadrics of projective dimension $s$ in $\PP^n$ are easily associated to $3$-dimensional tensors
of type $(s+1)\times (n+1)\times (n+1)$. When $s=n$, we obtain cubic $3$-dimensional tensors. The target of the paper is to determine relations between geometric properties of the Weddle locus of $\mathcal L$ and properties of the associated tensors, \textit{e.g.} the rank.

It is a general non-sense that geometric properties of Weddle loci determine properties of the associated tensors, and vice-versa. From the point of view of tensors, a similar approach (dropping any assumption on partial symmetry for the tensor) can be found also in the classical book of Gelfand, Kapranov, and Zelevinsky (\cite{GKZ94}), where they answer the question about which cubic curves arise from these general tensors. Our target is  to restrict our attention to partially symmetric tensors, and find an answer to similar type of questions.

There are two reasons for which we are led to restrict our attention to partially symmetric tensors. One is the mentioned link with linear systems of hypersurfaces, and thus with their related interpolation properties. There is also the observation that partially symmetric tensors arise naturally in standard decompositions of tensors spaces in supplementary subspaces, including the subspaces of symmetric and skew-symmetric tensors. The remaining summands are composed by special partially symmetric tensors, whose Weddle loci are studied, in the $3$-dimensional case, in the last section of the paper.

The main results that we obtain establish a connection between the singularities of the Weddle locus and the rank of the partially symmetric tensor associated to a linear system of quadrics. These results are focused on tensors of low rank, thus are exhaustive for linear systems of quadrics in low dimensional projective spaces. In the case  of $\PP^2$ we prove that when the $3\times 3\times 3$ tensor associated to a net of conics has rank $3$, then its Weddle locus splits in the union of three lines. When the tensor has rank $4$ then the net is generic. We prove that a general cubic curve in $\PP^2$ is the Weddle locus of a net of conics, extending to the partially symmetric case a result of Gelfand, Kapranov, and Zelevinsky.

We consider then linear systems of projective dimension $3$ of quadrics in $\PP^3$, whose associated tensor $T$ has type $4\times 4 \times 4$. For generic tensors  of this type, the Weddle locus is a quartic surface, which cannot be generic since it comes as the determinant of a matrix of linear forms (see \cite{Bea00}, and \cite{CGer14}).  Again, we prove that when  $T$ has rank $4$, then the Weddle locus of the linear system splits in the union of four planes. Since the generic rank of a tensor of type $4\times 4 \times 4$ is $6$, we pay special attention to systems of quadrics whose associated tensor has rank $5$. We prove that when the partially symmetric tensor $T$ has rank $5$, then the corresponding Weddle locus has $10$ singular points. It follows that when the singularities of the Weddle locus of $T$ have cardinality smaller than $10$ (as in the classical case of the tensor associated to a linear system of quadrics with $6$ general base points), then the rank of $T$ is at least $6$. We find in this sense a method
to exclude that a partially symmetric tensor has rank smaller than the generic value.

In the last section we consider the singularities of the Weddle locus of a partially symmetric tensor $n\times n \times n$ in the summand $N_1$ of the standard decomposition of tensors of dimension $3$ introduced by Itin and Reches in \cite{IR22}. The result shows that the singular locus of a general element of $N_1$ has cardinality equal to a recursively defined value, the $n$-th Jacobstal number $J_n$ (see \cite{Hor96}).

\subsubsection*{Notation}
We will work over the complex field $\mathbb C$, while many results are also valid over any algebraically closed field  of characteristic $0$.

	Let $V$ be a vector space of dimension $n+1$ over $\C$.
	From now on, we assume that a (ordered) basis $\mathbf{e} = \{e_0,\dots,e_n\}$ is fixed on $V$.
	This choice allows us to identify $V$ with its dual space $V^{\vee}$ and, consequently, the $d$-th tensor power $(V^{\vee})^{\otimes d}$ with $V^{\otimes d}$.
	In what follows, we always consider $d$-th order \emph{cubic} tensors, \textit{i.e.}, elements of $V^{\otimes d}$.
	They can always be regarded indistinctly as multilinear $d$-forms on $V$ or on $V^{\vee}$, thanks to the identification given by $\mathbf{e}$.
	In particular, a tensor $T \in V^{\otimes d}$ is written as
	\[
	T = \sum_{0 \le i_1,\dots,i_d \le n} T_{i_1 \dots i_d} \ e_{i_1} \otimes \dots \otimes e_{i_d}
	\]
	with respect to the basis $\mathbf{e}$ of $V$.
	
	Given $x \in V \setminus \{ 0 \}$ of coordinates $\mathbf{x} = (x_0,\dots,x_n)$ with respect to $\mathbf{e}$, we denote by $[x_0:\dots:x_n]$ the homogeneous coordinates relative to $\mathbf{e}$ of the corresponding point $ [x] \in \PP(V)$.
	Finally, if $V = \C^{n+1}$, the projective space $\PP(\C^{n+1})$ is simply denoted by $\PP^n$.

\section{Partially symmetric tensors}\label{sect_semisym_tensors}
	Let $S_d$ denote the symmetric group on $d$ elements.
	A cubic tensor $T$ in $V^{\otimes d}$ is called \emph{symmetric} if it satisfies
	\[
	T_{i_1 \dots i_d} = T_{i_{\sigma(1)} \dots i_{\sigma(d)}}
	\]
	for each multi-index $(i_1,\dots,i_d)$ and each permutation $\sigma \in S_d$.
	The subspace of {symmetric} tensors of $V^{\otimes d}$ is denoted by ${\operatorname{S}}^d V$.
	If $\dim V = n+1$, the dimension of the space of {symmetric} tensors is
	\begin{equation*}
		\dim {\operatorname{S}}^d V = \binom{n+d}{n} = \binom{n+d}{d} .
	\end{equation*}
	Recall that the \emph{symmetrization operator} $\operatorname{S}$ on $V^{\otimes d}$ is defined, in coordinates, as
	\begin{equation*}
		{\operatorname{S}(T)}_{i_1 \dots i_d} := \frac{1}{d!} \sum_{\sigma \in S_d} T_{i_{\sigma(1)} \dots i_{\sigma(d)}} .
	\end{equation*}
	The tensor $\operatorname{S}(T) \in {\operatorname{S}}^d V$ is called the \emph{symmetric part} of $T$.
	\begin{remark}\label{rem_sym_properties}
		Let $\operatorname{Im}\operatorname{S}$ denote the image of the symmetrization operator.
		It is well known that:
		\begin{itemize}
			\item $\operatorname{Im}\operatorname{S} = {\operatorname{S}}^d V$;
			\item $\operatorname{S}(T) = T$ if and only if $T \in {\operatorname{S}}^d V$.
		\end{itemize}
	\end{remark}
	
	A tensor $T \in V^{\otimes d}$ is called \emph{skew-symmetric} if it satisfies
	\[
	T_{i_1 \dots i_d} = \operatorname{sgn}(\sigma) T_{i_{\sigma(1)} \dots i_{\sigma(d)}}
	\]
	for each multi-index and each permutation $\sigma$, where $\operatorname{sgn}(\sigma)$ denotes the sign of $\sigma$.
	Recall that the subspace $\Lambda^d V$ of {skew-symmetric} tensors of $V^{\otimes d}$ has dimension
	\begin{equation*}
		\dim \Lambda^d V = \binom{n+1}{d} .
	\end{equation*}
	The \emph{skew-symmetrization operator} operator $\operatorname{A}$ on $V^{\otimes d}$ is defined as
	\begin{equation*}
		{\operatorname{A}(T)}_{i_1 \dots i_d} := \frac{1}{d!} \sum_{\sigma \in S_d} \operatorname{sgn}(\sigma) T_{i_{\sigma(1)} \dots i_{\sigma(d)}} ,
	\end{equation*}
	and $\operatorname{A}(T) \in \Lambda^d V$ is called the \emph{skew-symmetric part} of $T$.
	\begin{remark}\label{rem_skw_properties}
		Analogously to Remark~\ref{rem_sym_properties}, we recall the following well known facts:
		\begin{itemize}
			\item $\operatorname{Im}\operatorname{A} = \Lambda^d V$;
			\item $\operatorname{A}(T) = T$ if and only if $T \in \Lambda^d V$;
			\item $\operatorname{S}(T) = 0$ if $T \in \Lambda^d V$;
			\item $\operatorname{A}(T) = 0$ if $T \in {\operatorname{S}}^d V$.
		\end{itemize}
	\end{remark}
	
	It is clear that a $2$-nd order {cubic tensor} $T$, \textit{i.e.}, a matrix, can be always written as the sum of its symmetric and skew-symmetric parts as
	\begin{equation*}
		T = \operatorname{S}(T) + \operatorname{A}(T)
	\end{equation*}
	where, in this case,
	\[
	{\operatorname{S}(T)}_{ij} = \frac{1}{2} \big( T_{ij} + T_{ji} \big) \quad \text{and} \quad {\operatorname{A}(T)}_{ij} = \frac{1}{2} \big( T_{ij} - T_{ji} \big) .
	\]
	Accordingly, the space of $2$-nd order cubic tensors decomposes as
	\begin{equation*}
		V^{\otimes 2} = {\operatorname{S}}^2 V \oplus \Lambda^2 V
	\end{equation*}
	and the dimensions of these spaces are distributed as
	\[
	(n+1)^2 = \binom{n+2}{2} + \binom{n+1}{2}
	\]
	respectively.
	
	For dimensional reasons, a decomposition as above does not exist for higher order tensors.
	Specifically, alongside their symmetric and skew-symmetric components, higher order tensors always have a \emph{residual part}.
	In the next section, we focus on cubic tensors of $3$-rd order and we investigate their residual part.
	
\subsection{Decomposition of $3$-rd order cubic tensors}
	Let $T \in V^{\otimes 3}$ be a $3$-rd order tensor, and let $\{T_{ijk}\}_{0 \le i,j,k\le n}$ be its components with respect to the basis $\mathbf{e}$ of $V$ fixed above.
	First, we fix a convention in order to arrange these values in a ``three dimensional array''.
	\begin{example}\label{ex_2x2x2_tensors}
		If $\dim V = 2$, we represent the tensor $T$ by a $2 \times 2 \times 2$ array as follows.
		\begin{center}
			\begin{tikzpicture}[scale=1]
				\draw (0,0) node[scale=1] {$T_{100}$};
				\draw (2,0) node[scale=1] {$T_{110}$};
				\draw (1,1) node[scale=1] {$T_{101}$};
				\draw (3,1) node[scale=1] {$T_{111}$};
				\draw (0,2) node[scale=1] {$T_{000}$};
				\draw (2,2) node[scale=1] {$T_{010}$};
				\draw (1,3) node[scale=1] {$T_{001}$};
				\draw (3,3) node[scale=1] {$T_{011}$};
				
				\draw (0,0.3) -- (0,1.7);
				\draw (0.3,0) -- (1.7,0);
				\draw (2,0.3) -- (2,1.7);
				\draw (0.3,2) -- (1.7,2);
				\draw (1,1.3) -- (1,2.7);
				\draw (1.3,1) -- (2.7,1);
				\draw (3,1.3) -- (3,2.7);
				\draw (1.3,3) -- (2.7,3);
				\draw (0.2,0.2) -- (0.8,0.8);
				\draw (0.2,2.2) -- (0.8,2.8);
				\draw (2.2,0.2) -- (2.8,0.8);
				\draw (2.2,2.2) -- (2.8,2.8);
				
			\end{tikzpicture}
		\end{center}
		This representation adopts the index convention of the following Remark~\ref{rem_convention_indices}.
	\end{example}
	\begin{remark}\label{rem_convention_indices}
		If $\dim V = n+1$, for each $0 \le k \le n$, we define a matrix $T_k$ by setting $(T_k)_{ij} := T_{ijk}$, \textit{i.e.}, the component $T_{ijk}$ of $T$ is the $(i,j)$ entry of $T_k$, where $i$, resp. $j$, is the row, resp. column, index.
		We view the collection of matrices $\{T_k\}_{0\le k\le n}$ as forming a $(n+1) \times (n+1) \times (n+1)$ array, where $T_0$ represents the first (or front) face of this array, $T_1$ the second face parallel to the first, and so on, up to $T_n$, which represents the $(n+1)$-th (or back) face.
		
		With a similar convention, we refer to the indices $i,j$ and $k$ as the up-down, left-right, and front-back directions respectively.
	\end{remark}
	
	The main object of our investigation is the linear subspace of tensors in $V^{\otimes 3}$ defined by the set of linear equations
	\begin{equation*}
		T_{ijk} + T_{jki} + T_{kij} = 0, \quad i,j,k = 0,\dots,n .
	\end{equation*}
	We denote this space by ${\operatorname{N}} V$.
	
	It is useful to describe ${\operatorname{N}} V$ as the image of the following linear map.
	Specifically, we define the operator $\operatorname{N}$ on $V^{\otimes 3}$ by
	\begin{equation*}
		{\operatorname{N}(T)} := T - {\operatorname{S}(T)} - {\operatorname{A}(T)} ,
	\end{equation*}
	see also \cite[Section~4.1.1]{IR22}.
	Since, in coordinates, the symmetric part of $T$ is
	\begin{equation*}
		{\operatorname{S}(T)}_{ijk} = \frac{1}{6} \big( T_{ijk} + T_{jki} + T_{kij} + T_{jik} + T_{kji} + T_{ikj} \big)
	\end{equation*}
	and its skew-symmetric part is
	\begin{equation*}
		{\operatorname{A}(T)}_{ijk} = \frac{1}{6} \big( T_{ijk} + T_{jki} + T_{kij} - T_{jik} - T_{kji} - T_{ikj} \big) ,
	\end{equation*}
	the components of $\operatorname{N}(T)$ are given by
	\begin{equation}\label{eq_residue_part}
		{\operatorname{N}(T)}_{ijk} = {\big( T - \operatorname{S}(T) - \operatorname{A}(T) \big)}_{ijk} = \frac{1}{3} \big( 2T_{ijk} - T_{jki} - T_{kij} \big).
	\end{equation}
	The tensor $\operatorname{N}(T)$ is called the \emph{residual part} of $T$, \textit{cf.} \cite[p.~9]{IR22}.
	
	\begin{proposition}\label{prop_imN=NV}
		The equality $\operatorname{Im}\operatorname{N} = {\operatorname{N}} V$ holds.
		
		Moreover, for $T \in V^{\otimes 3}$, $\operatorname{N}(T) = T$ if and only if $T \in {\operatorname{N}} V$.
	\end{proposition}
	\begin{proof}
		If $T = \operatorname{N}(R)$ for some $R \in V^{\otimes 3}$ then, by Equation~\eqref{eq_residue_part}, it follows that
		\begin{align*}
			3 (T_{ijk} &+ T_{jki} + T_{kij}) \\
			&= \big( 2R_{ijk} - R_{jki} - R_{kij} \big) + \big( 2R_{jki} - R_{kij} - R_{ijk} \big) +\big( 2R_{kij} - R_{ijk} - R_{jki} \big) \\
			&= 0 ,
		\end{align*}
		hence $T \in {\operatorname{N}} V$.
		
		The converse follows immediately by rewriting Equation~\eqref{eq_residue_part} as
		\[
		{\operatorname{N}(T)}_{ijk} = T_{ijk} - \frac{1}{3} \big( T_{ijk} + T_{jki} + T_{kij} \big) .
		\]
		Indeed, if $T \in {\operatorname{N}} V$, we get directly that $\operatorname{N}(T) = T$.
	\end{proof}
	
	\begin{proposition}\label{prop_decomposition_3_order}
		The space of $3$-rd order cubic tensors decomposes as
		\begin{equation*}
			V^{\otimes 3} = {\operatorname{S}}^3 V \oplus {\operatorname{N}} V \oplus \Lambda^3 V .
		\end{equation*}
		Accordingly, the dimensions of these spaces are distributed as
		\begin{equation}\label{eq_3_order_decomp_dim}
			(n+1)^3 = \binom{n+3}{3} + 4 \binom{n+2}{3} + \binom{n+1}{3}
		\end{equation}
		respectively.
	\end{proposition}
	\begin{proof}
		It holds that $ \ker\operatorname{N} = {\operatorname{S}}^3 V \oplus \Lambda^3 V$.
		Indeed, if $T \in \ker\operatorname{N}$ then $T = \operatorname{S}(T) + \operatorname{A}(T)$ by definition of $\operatorname{N}$.
		Conversely, if $T = T_1 + T_2$ for some tensors $T_1 \in {\operatorname{S}}^3 V$ and $T_2 \in \Lambda^3 V$, we have
		\begin{equation*}
			\operatorname{N}(T) = T_1 - \operatorname{S}(T_1) - \operatorname{A}(T_1) + T_2 - \operatorname{S}(T_2) - \operatorname{A}(T_2) = 0 ,
		\end{equation*}
		where the last equality follows from Remarks~\ref{rem_sym_properties} and~\ref{rem_skw_properties}.
		
		From this, we compute
		\[
		\dim \operatorname{Im}\operatorname{N} = \dim V^{\otimes 3} - \dim {\operatorname{S}}^3 V - \dim \Lambda^3 V = \frac{2}{3}n(n+1)(n+2) .
		\]
		Since $\operatorname{Im}\operatorname{N} = {\operatorname{N}} V$ by Proposition~\ref{prop_imN=NV}, we get Equation~\eqref{eq_3_order_decomp_dim}.
		
		Finally, if $ T \in \operatorname{Im}\operatorname{N} \cap \ker\operatorname{N} $, then $ T = N(T) = 0 $, where the first equality follows by Proposition~\ref{prop_imN=NV} once more.
		This proves that
		\[
		{\operatorname{N}} V \cap \big({\operatorname{S}}^3 V \oplus \Lambda^3 V\big) = \{ 0 \} ,
		\]
		which establishes the desired decomposition of $V^{\otimes 3}$.
	\end{proof}
	
	\begin{remark}\label{rem_focus_partially_symmetric}
	As we will clarify in Section~\ref{sect_weddle_loci}, our goal is to provide a concrete geometric interpretation of a tensor $T \in {\operatorname{N}} V$.
	Following the index convention of Remark~\ref{rem_convention_indices}, we regard $T$ as a collection of matrices $\{T_k\}_k$, which form the faces of $T$.
	Observe that if $T$ were \emph{partially symmetric with respect to the first two indices}, then each $T_k$ would be symmetric, hence $T$ would represent a \emph{linear system of quadrics}.
	
	\end{remark}
	From now on, we focus our attention on the tensors in ${\operatorname{N}} V$ that are also partially symmetric.
	
	Denote by ${\operatorname{N}_1} V$ the subspace of ${\operatorname{N}} V$ defined by the linear equations
	\begin{enumerate}
		\item $T_{ijk} + T_{jki} + T_{kij} = 0 ,$
		\item $T_{ijk} = T_{jik} .$
	\end{enumerate}
	In a similar fashion, we denote by ${\operatorname{N}_2} V$ the subspace of tensors in ${\operatorname{N}} V$ that are partially symmetric with respect to the first and last indices, \textsl{i.e.}, $T_{ijk} = T_{kji}$.
	
	Now, split the operator $\operatorname{N}$ into the sum of two operators $\operatorname{N}_1$ and $\operatorname{N}_2$ on $V^{\otimes 3}$, which are defined as
	\[
	{\operatorname{N}_1 (T)}_{ijk} := \frac{T_{ijk} + T_{jik} - T_{kji} - T_{kij}}{3}
	\]
	and
	\[
	{\operatorname{N}_2 (T)}_{ijk} := \frac{T_{ijk} - T_{jik} + T_{kji} - T_{jki}}{3}
	\]
	respectively, see \cite[Section~4.1.2]{IR22}.
	
	\begin{proposition}\label{prop_imN1=N1V-imN2=N2V}
		The following properties hold:
		\begin{enumerate}
			\item $\operatorname{Im}\operatorname{N}_1 = {\operatorname{N}_1} V$ and $\operatorname{Im}\operatorname{N}_2 = {\operatorname{N}_2} V$;\label{item:N1_N2_i}
			\item ${\operatorname{N}_1} V \cap {\operatorname{N}_2} V = \{ 0 \}$;\label{item:N1_N2_ii}
			\item ${\operatorname{N}_1} V \subseteq \ker \operatorname{N}_2$ and ${\operatorname{N}_2} V \subseteq \ker \operatorname{N}_1$;\label{item:N1_N2_iii}
		\end{enumerate}
		moreover, for $T \in V^{\otimes 3}$,
		\begin{enumerate}[resume]
			\item $\operatorname{N}_1(T) = T$ if and only if $T \in {\operatorname{N}_1} V$;\label{item:N1_N2_iv}
			\item $\operatorname{N}_2(T) = T$ if and only if $T \in {\operatorname{N}_2} V$.\label{item:N1_N2_v}
		\end{enumerate}
	\end{proposition}
	\begin{proof}
		The proof of item~\eqref{item:N1_N2_i} is similar to the proof of Proposition~\ref{prop_imN=NV}, and follows from an easy computation.
		
		To prove item~\eqref{item:N1_N2_ii}, suppose that $T \in {\operatorname{N}_1} V \cap {\operatorname{N}_2} V$.
		This means that $T$ satisfies the cyclic condition $T_{ijk} + T_{jki} + T_{kij} = 0$ and is simultaneously partially symmetric with respect to both the first and second indices, as well as the first and last indices.
		From these properties we derive the following chain of equalities
		\begin{align*}
			T_{ijk} = - T_{jki} - T_{kij} = - T_{kji} - T_{jik} = - T_{ijk} - T_{ijk}
		\end{align*}
		which yields $3 T_{ijk} = 0$ for each $(i,j,k)$ entry, and thus $T=0$.
		
		To show item~\eqref{item:N1_N2_iii}, let $T \in {\operatorname{N}_1} V$.
		Then, we have
		\[
		3{\operatorname{N}_2 (T)}_{ijk} = T_{ijk} - T_{jik} + T_{kji} - T_{jki} = T_{ijk} - T_{ijk} + T_{kji} - T_{kji} = 0
		\]
		where in the second equality we used the partial symmetry of $T$ with respect to the first and second indices.
		The proof of the remaining inclusion is similar.
		
		To conclude, we show item~\eqref{item:N1_N2_iv} (the proof of item~\eqref{item:N1_N2_v} is completely analogous).
		For $T \in {\operatorname{N}_1} V$, we have
		\[
		3{\operatorname{N}_1 (T)}_{ijk} = T_{ijk} + T_{jik} - T_{kji} - T_{kij} = T_{ijk} + T_{ijk} - (T_{jki} + T_{kij}) = 3 T_{ijk} ,
		\]
		hence $\operatorname{N}_1 (T) = T$.
		The converse follows from item~\eqref{item:N1_N2_i}.
	\end{proof}
	
	Propositions~\ref{prop_imN=NV} and~\ref{prop_imN1=N1V-imN2=N2V} allows us to conclude that the following direct sum decomposition
	\begin{equation}\label{eq_splitting_residual_tensors}
		{\operatorname{N}} V = {\operatorname{N}_1} V \oplus {\operatorname{N}_2} V
	\end{equation}
	holds.
	
	\begin{remark}
	The operators $\operatorname{N}_1$ and $\operatorname{N}_2$ arise naturally from representation theory.
	Indeed, they are defined by the two standard Young tableaux associated to the non-increasing partition $(2,1,0)$ of the integer $3$ in parts less or equal than $3$.
	The non uniqueness of the decomposition~\eqref{eq_splitting_residual_tensors} given by the operators $\operatorname{N}_1$ and $\operatorname{N}_2$ is, for instance, due to the fact that the rule we take into account for the explicit expression of these two operators is that the symmetrizers are applied after the skew-symmetrizers.
	Other choices for the rules of the Young symmetrizers associated to $(2,1,0)$ give two different operators that, however, are isotopic to $\operatorname{N}_1$ and $\operatorname{N}_2$.
	For further details see, for instance, \cite[p.~80]{Ful97} and \cite[Section~4.1.2]{IR22}.
	\end{remark}
	
	\begin{remark}
	According to this interpretation (see \cite[Section~7]{Ful97}), recall that the symmetrization operator $\operatorname{S}$ corresponds to the partition $(3,0,0)$, hence the Schur power $\Gamma^{(3,0,0)} V$ is isomorphic to the symmetric power ${\operatorname{S}}^3 V$.
	Analogously, the other extreme partition $(1,1,1)$ is associated to the skew-symmetrization operator $\operatorname{A}$, thus the Schur power $\Gamma^{(1,1,1)} V$ is isomorphic to the exterior power ${\Lambda^3} V$.
	Given that the representation associated to $(2,1,0)$ is two-dimensional, we have two isomorphic representations for the Schur power $\Gamma^{(2,1,0)} V$ which are the spaces ${\operatorname{N}_1} V$ and ${\operatorname{N}_2} V$.
	\end{remark}
	
\subsubsection{Interpretation in terms of bases}
	By Proposition~\ref{prop_decomposition_3_order} and Equation~\eqref{eq_splitting_residual_tensors}, we have decomposed the identity operator on $V^{\otimes 3}$ as
		\begin{equation*}
			{\operatorname{Id}_{V^{\otimes 3}} = \operatorname{S} + (\operatorname{N}_1 + \operatorname{N}_2) + \operatorname{A}}
		\end{equation*}
	and this allows to describe tensors in ${\operatorname{N}} V$ more concretely in terms of bases.
	
	If $ \{ e_0,\dots,e_n \} $ is basis for $V$, then $ \{ e_i \otimes e_j \otimes e_k \}_{i,j,k} $ is the basis induced on $V^{\otimes 3}$.
	First, consider the partition of the set
	\[
	\big\{ (i,j,k) \ \text{s.t.} \ 0 \le i,j,k \le n \big\}
	\]
	given by
	\begin{center}
		$\{j \le i \le k\}$, $\{j > i > k\}$, $\{j \le i > k\}$, $\{j > i \le k\}$.
	\end{center}
	As usual, a basis for the space ${\operatorname{S}}^3 V$ is
	\[
	\big\{{e_i \odot e_j \odot e_k := \operatorname{S}(e_i \otimes e_j \otimes e_k)}\big\}_{j \le i \le k} ,
	\]
	while a basis for $\Lambda^3 V$ is
	\[
	\big\{{e_i \wedge e_j \wedge e_k := \operatorname{A}(e_i \otimes e_j \otimes e_k)}\big\}_{j > i > k} .
	\]
	Finally, in order to find bases for ${\operatorname{N}_1} V$, resp. ${\operatorname{N}_2} V$, we apply the operator $\operatorname{N}_1$, resp. $\operatorname{N}_2$, to the elements $e_i \otimes e_j \otimes e_k$ whose indices satisfy the relation $j \le i > k$, resp. $j > i \le k$.
	More precisely, one can check that
	\begin{equation*}
		\big\{{\operatorname{N}_1(e_i \otimes e_j \otimes e_k)}\big\}_{j \le i > k} \quad \text{and} \quad \big\{{\operatorname{N}_2(e_i \otimes e_j \otimes e_k)}\big\}_{j > i \le k}
	\end{equation*}
	are two sets of linearly independent tensors which, by a dimension count, form bases of ${\operatorname{N}_1} V$ and ${\operatorname{N}_2} V$ respectively.
	
	\begin{example}\label{ex_N1_dimension_2}
		If $\dim V = 2$, then the tensors
		\[
		e_0 \otimes e_0 \otimes e_0,\ e_0 \otimes e_0 \otimes e_1,\ e_1 \otimes e_0 \otimes e_1,\ e_1 \otimes e_1 \otimes e_1
		\]
		are mapped by the operator $S$ to the basis
		\begin{gather*}
			e_0 \odot e_0 \odot e_0 = e_{{0}} \otimes e_{{0}} \otimes e_{{0}} ,\\
			e_0 \odot e_0 \odot e_1 = \frac{1}{3} (e_{{0}} \otimes e_{{0}} \otimes e_{{1}} + e_1 \otimes e_0 \otimes e_0 + e_0 \otimes e_1 \otimes e_0) ,\\
			e_1 \odot e_0 \odot e_1 = \frac{1}{3} (e_{{1}} \otimes e_{{0}} \otimes e_{{1}} + e_1 \otimes e_1 \otimes e_0 + e_0 \otimes e_1 \otimes e_1) ,\\
			e_1 \odot e_1 \odot e_1 = e_{{1}} \otimes e_{{1}} \otimes e_{{1}}
		\end{gather*}
		of the space ${\operatorname{S}}^3 V$.
		
		Instead, the operator $\operatorname{N}_1$ maps
		\[
		e_1 \otimes e_0 \otimes e_0,\ e_1 \otimes e_1 \otimes e_0
		\]
		to the basis of ${\operatorname{N}_1} V$ formed by the tensors
		\begin{gather*}
			\frac{1}{3} (e_{{1}} \otimes e_{{0}} \otimes e_{{0}}  + e_0 \otimes e_1 \otimes e_0 - 2 e_0 \otimes e_0 \otimes e_1) ,\\
			\frac{1}{3} (2 e_{{1}} \otimes e_{{1}} \otimes e_{{0}} - e_1 \otimes e_0 \otimes e_1 - e_0 \otimes e_1 \otimes e_1)
		\end{gather*}
		which, in the notation of Example~\ref{ex_2x2x2_tensors}, are represented as
		\begin{center}
			\begin{tikzpicture}[scale=0.8]
				\draw (0,0) node[scale=1] {$\frac{1}{3}$};
				\draw (2,0) node[scale=1] {$0$};
				\draw (1,1) node[scale=1] {$0$};
				\draw (3,1) node[scale=1] {$0$};
				\draw (0,2) node[scale=1] {$0$};
				\draw (2,2) node[scale=1] {$\frac{1}{3}$};
				\draw (1,3) node[scale=1] {$-\frac{2}{3}$};
				\draw (3,3) node[scale=1] {$0$};
				
				\draw (0,0.3) -- (0,1.7);
				\draw (0.3,0) -- (1.7,0);
				\draw (2,0.3) -- (2,1.7);
				\draw (0.3,2) -- (1.7,2);
				\draw (1,1.3) -- (1,2.7);
				\draw (1.3,1) -- (2.7,1);
				\draw (3,1.3) -- (3,2.7);
				\draw (1.3,3) -- (2.7,3);
				\draw (0.2,0.2) -- (0.8,0.8);
				\draw (0.2,2.2) -- (0.8,2.8);
				\draw (2.2,0.2) -- (2.8,0.8);
				\draw (2.2,2.2) -- (2.8,2.8);
				
			\end{tikzpicture}
			$ \quad $
			\begin{tikzpicture}[scale=0.8]
			\draw (0,0) node[scale=1] {$0$};
			\draw (2,0) node[scale=1] {$\frac{2}{3}$};
			\draw (1,1) node[scale=1] {$-\frac{1}{3}$};
			\draw (3,1) node[scale=1] {$0$};
			\draw (0,2) node[scale=1] {$0$};
			\draw (2,2) node[scale=1] {$0$};
			\draw (1,3) node[scale=1] {$0$};
			\draw (3,3) node[scale=1] {$-\frac{1}{3}$};
			
			\draw (0,0.3) -- (0,1.7);
			\draw (0.3,0) -- (1.7,0);
			\draw (2,0.3) -- (2,1.7);
			\draw (0.3,2) -- (1.7,2);
			\draw (1,1.3) -- (1,2.7);
			\draw (1.3,1) -- (2.7,1);
			\draw (3,1.3) -- (3,2.7);
			\draw (1.3,3) -- (2.7,3);
			\draw (0.2,0.2) -- (0.8,0.8);
			\draw (0.2,2.2) -- (0.8,2.8);
			\draw (2.2,0.2) -- (2.8,0.8);
			\draw (2.2,2.2) -- (2.8,2.8);
			
			\end{tikzpicture}
		\end{center}
		respectively.
	\end{example}
	
	\begin{example}\label{ex_N1_dimension_3}
		If $\dim V = 3$, we see from Equation~\eqref{eq_3_order_decomp_dim} that $\dim {\operatorname{N}_1} V = 8$.
		In this case, the tensors $\big\{ 3 {e_i \otimes e_j \otimes e_k} \big\}_{j \le i > k}$ are mapped by the operator ${\operatorname{N}_1}$ to the basis
		\begin{center}
			$ e_{{1}} \otimes e_{{0}} \otimes e_{{0}}  + e_0 \otimes e_1 \otimes e_0 - 2 e_0 \otimes e_0 \otimes e_1  ,$
			
			$ e_{{2}} \otimes e_{{0}} \otimes e_{{0}} + e_0 \otimes e_2 \otimes e_0 - 2 e_0 \otimes e_0 \otimes e_2 ,$
			
			$ e_{{2}} \otimes e_{{1}} \otimes e_{{1}} + e_1 \otimes e_2 \otimes e_1 - 2 e_1 \otimes e_1 \otimes e_2 ,$

			$ 2 e_{{1}} \otimes e_{{1}} \otimes e_{{0}} - e_1 \otimes e_0 \otimes e_1 - e_0 \otimes e_1 \otimes e_1 ,$
			
			$ 2 e_{{2}} \otimes e_{{2}} \otimes e_{{0}} - e_2 \otimes e_0 \otimes e_2 - e_0 \otimes e_2 \otimes e_2 ,$
			
			$ 2 e_{{2}} \otimes e_{{2}} \otimes e_{{1}} - e_2 \otimes e_1 \otimes e_2 - e_1 \otimes e_2 \otimes e_2 ,$
			
			$ e_{{2}} \otimes e_{{1}} \otimes e_{{0}} + e_1 \otimes e_2 \otimes e_0 - e_0 \otimes e_1 \otimes e_2 - e_1 \otimes e_0 \otimes e_2 ,$
			
			$ e_{{2}} \otimes e_{{0}} \otimes e_{{1}} + e_0 \otimes e_2 \otimes e_1 - e_0 \otimes e_1 \otimes e_2 - e_1 \otimes e_0 \otimes e_2 $
		\end{center}
		of the space ${\operatorname{N}_1} V$.
	\end{example}
	
\subsection{Cyclic-symmetric tensors}
	As already noted in Remark~\ref{rem_focus_partially_symmetric}, according to the index convention adopted in Remark~\ref{rem_convention_indices}, it is natural to focus on tensors that are partially symmetric with respect to the first two indices when representing linear systems of quadrics.
	Furthermore, as the spaces ${\operatorname{N}_1} V$ and ${\operatorname{N}_2} V$ are isomorphic, we will restrict our attention to tensors in ${\operatorname{N}_1} V$ from now on.
	Moreover, since a tensor $T \in {\operatorname{N}_1} V$ satisfies also the cyclic condition $T_{ijk} + T_{jki} + T_{kij} = 0$, we give the following.
	\begin{definition}
		We call a tensor $T$ belonging to ${\operatorname{N}_1} V$ a \emph{cyclic-symmetric} tensor.
		In the following, it is always understood that the partial symmetry of $T$ is with respect to its first and second indices.
	\end{definition}
	
	We give here below the general explicit expression of cyclic-symmetric tensors in low dimension.
	\begin{example}\label{ex_residue_tensor_dim2}
		If $\dim V = 2$, then $\dim {\operatorname{N}_1} V = 2$.
		In this case, a general cyclic-symmetric tensor has the form:
		\begin{center}
			\begin{tikzpicture}[scale=1]
				\draw (0,0) node[scale=1] {$a$};
				\draw (2,0) node[scale=1] {$2b$};
				\draw (1,1) node[scale=1] {$-b$};
				\draw (3,1) node[scale=1] {$0$};
				\draw (0,2) node[scale=1] {$0$};
				\draw (2,2) node[scale=1] {$a$};
				\draw (1,3) node[scale=1] {$-2a$};
				\draw (3,3) node[scale=1] {$-b$};
				
				\draw (0,0.3) -- (0,1.7);
				\draw (0.3,0) -- (1.7,0);
				\draw (2,0.3) -- (2,1.7);
				\draw (0.3,2) -- (1.7,2);
				\draw (1,1.3) -- (1,2.7);
				\draw (1.3,1) -- (2.7,1);
				\draw (3,1.3) -- (3,2.7);
				\draw (1.3,3) -- (2.7,3);
				\draw (0.2,0.2) -- (0.8,0.8);
				\draw (0.2,2.2) -- (0.8,2.8);
				\draw (2.2,0.2) -- (2.8,0.8);
				\draw (2.2,2.2) -- (2.8,2.8);
				
			\end{tikzpicture}
		\end{center}
		\textit{cf.} the last part of Example~\ref{ex_N1_dimension_2}.
	\end{example}
	\begin{example}\label{ex_residue_tensor_dim3}
		If $\dim V = 3$, then, by Equation~\eqref{eq_3_order_decomp_dim}, $\dim {\operatorname{N}_1} V = 8$.
		The faces of a cyclic-symmetric tensor are arranged as follows:
		\begin{center}
			\includegraphics[scale=0.7]{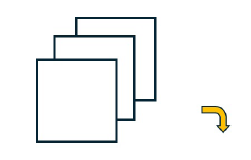}
		\end{center}
		\begin{equation*}
			\begin{bmatrix}
				0 & a & c \\
				a & 2b & d \\
				c & d & 2e
			\end{bmatrix}
			\begin{bmatrix}
				-2a & -b & f \\
				-b & 0 & g \\
				f & g & 2h
			\end{bmatrix}
			\begin{bmatrix}
				-2c & -d-f & -e \\
				-d-f & -2g & -h \\
				-e & -h & 0
			\end{bmatrix}
		\end{equation*}
		\textit{cf.} Example~\ref{ex_N1_dimension_3}.
	\end{example}
	
	We observe immediately that the cyclic-symmetric tensor in Example~\ref{ex_residue_tensor_dim2} is a \emph{sub-tensor} (see \cite[Definition~6.4.2]{BC19} for the definition) of the one in Example~\ref{ex_residue_tensor_dim3}.
	More precisely, in the notation of Remark~\ref{rem_convention_indices}, the former tensor is the upper-left-front sub-tensor of the latter.
	We generalize this construction in any dimension thanks to the following two results.
	
	\begin{proposition}\label{thm_construction_semisym_tensors_inverse}
		Let $T$ be a tensor in ${\operatorname{N}_1} V \cong {\operatorname{N}_1} \C^{n+1}$.
		We obtain a sub-tensor $S \in {\operatorname{N}_1} \C^{n}$ by removing from $T$ its last face $T_n$, as well as the last row and the last column from each face $T_k$, for $0 \le k < n$.
	\end{proposition}
	\begin{proof}
		It is clear that the new tensor $S$ constructed in this way is partially symmetric with respect to the first two indices.
		The only subtle point is to check that $S$ satisfy the relation
		\[
		S_{ijk} + S_{jki} + S_{kij} = 0
		\]
		for all $0\le i,j,k <n$.
		This easily follows as any entry $T_{ijk}$ where at least one of the indices $i,j,k$ equals $n$ lies entirely within the portions removed.
		Namely, it is either in the last row of the face $T_k$, in the last column of the face $T_k$ or in the last face $T_n$.
	\end{proof}
	We now prove that the previous proposition can be reversed.
	\begin{proposition}\label{thm_construction_semisym_tensors}
		Given any tensor $S \in {\operatorname{N}_1} \C^{n}$, it is always possible to find a tensor $T \in {\operatorname{N}_1} \C^{n+1}$ such that $S$ is a sub-tensor of $T$. Moreover, $S$ is obtained from $T$ by the procedure introduced in Proposition~\ref{thm_construction_semisym_tensors_inverse}.
	\end{proposition}
	\begin{proof}
		We prove this proposition by induction.
		
		In Examples~\ref{ex_residue_tensor_dim2} and~\ref{ex_residue_tensor_dim3} we discussed the base cases $n=1,2$.
		Therefore, suppose that a tensor $S$ in ${\operatorname{N}_1} \C^{n}$ is given.
		We build a tensor $T \in {\operatorname{N}_1} \C^{n+1}$ such that $S$ is a sub-tensor of $T$.
		Without loss of generality, we want $S$ to be the upper-left-front sub-tensor (in the sense of Remark~\ref{rem_convention_indices}) of the tensor $T$ we are going to construct.
		
		Hence, for $0 \le k < n$, the face $T_k$ has to be of the form
		\begin{equation*}
			\begin{bmatrix}
				& & & T_{0nk} \\
				& S_k& & \vdots \\
				& & & \vdots \\
				T_{0nk} & \dots & \dots & 2 T_{nnk}
			\end{bmatrix}
		\end{equation*}
		where the symmetric sub-matrix $S_k$ is the $k$-th face of $S$, and $T_{0nk}, \dots, T_{nnk}$ are new elements that can be chosen freely for each $k=0,\dots,n-1$.
		Indeed, note that these elements are not afflicted by the symmetries
		\begin{equation*}
			T_{ijk} + T_{jki} + T_{kij} = 0 \quad \text{for} \quad 0 \le k < n
		\end{equation*}
		of cyclic-symmetric tensors.
		Moreover, we set the entry $(T_k)_{nn}$ to be $2 T_{nnk}$ for later convenience.
		
		Therefore, it remains only to construct the last face $T_n$ of the tensor $T$.
		Note that the symmetries $T_{ijk} = - (T_{jki} + T_{kij})$ allow to express each entry $T_{ijn}$ of $T_n$ uniquely as a combination of the elements $T_{0nk}, \dots, T_{nnk}$, for each $k=0,\dots,n-1$, introduced above.
		Therefore, $T_n$ can be written as
		\begin{equation*}
			\begin{bmatrix}
				- 2 T_{0n0} & - (T_{0 n 1} + T_{1 n 0}) & \dots & - T_{nn0} \\
				- (T_{0 n 1} + T_{1 n 0}) & - 2 T_{1n1} & \dots & - T_{nn1} \\
				\vdots & \vdots & \ddots & \vdots \\
				- T_{nn0} & - T_{nn1} & \dots & 0
			\end{bmatrix} .
		\end{equation*}
		Summing up, from a given $S \in {\operatorname{N}_1} \C^{n}$, we built a new tensor $T \in {\operatorname{N}_1} \C^{n+1}$.
	\end{proof}
	
	\begin{remark}\label{rem_free}
		We observe that in the previous proof, we can freely choose the elements $T_{0nk}, \dots, T_{nnk}$, for each $k=0,\dots,n-1$, of the last columns of the first $n$ faces of $T$.
		The remaining elements of $T \setminus S$ are determined consequently.
	\end{remark}
	
	It is clear that to a cyclic-symmetric tensor as in Example~\ref{ex_residue_tensor_dim3} one can associate a special linear system of symmetric matrices (conics).
	For the same reason, in any dimension, there is a correspondence between tensors in ${\operatorname{N}_1} V$ and special linear system of quadrics, which we investigate Section~\ref{sect_weddle_semisym_tensors}.
	
	\begin{remark}
		Observe that, in the notation of Proposition~\ref{thm_construction_semisym_tensors}, the linear system of quadrics associated to the tensor $S$ corresponds to the linear system associated to $T$ cut with the hyperplane $\{ x_n = 0 \}$.
	\end{remark}

\section{Weddle loci}\label{sect_weddle_loci}
	We mentioned in Section~\ref{sect_semisym_tensors} that elements of ${\operatorname{N}_1} V$ are associated to special linear systems of quadrics, therefore they fit in the general algebraic theory of such systems.
	
	Following the convention of Remark~\ref{rem_convention_indices}, we focus on linear systems $\mathcal{L}$ of quadrics in $\PP(V)$ that can be described by cubic tensors, in the sense that $\mathcal{L}$ is the linear system generated by the quadrics represented by the faces $\{T_k\}_{k}$ of some $T \in V^{\otimes 3}$.
	For this, $T$ is {partially symmetric} with respect to the first two indices.
	In particular, recall that cyclic-symmetric tensors of the space ${\operatorname{N}_1} V$ have this symmetry.
	Moreover, the affine dimension of $\mathcal{L}$ is bounded by the dimension of the vector space $V$.
	
	Summing up, from now on we always make the following identification:
	\begin{center}
		\begin{tabular}{ c c c }
			$\big\{ \mathcal{L} \subset \PP({\operatorname{S}}^2 V) \ \text{s.t.} \ \dim \mathcal{L} \le \dim V \big\}$ &
			$\leftrightarrows$ &
			$\big\{ T \in V^{\otimes 3} \ \text{s.t.} \ T_{ijk} = T_{jik} \big\}$    
		\end{tabular}
	\end{center}
	
	\begin{remark}\label{rem_system_in_S+N}
		By Proposition~\ref{prop_decomposition_3_order} and the decomposition~\eqref{eq_splitting_residual_tensors}, we observe that any linear system of quadrics as above is represented by a tensor belonging to the space
		\[
		{\operatorname{S}}^3 V \oplus {\operatorname{N}_1} V .
		\]
	\end{remark}
	
\subsection{Weddle locus of a linear system}
	Let $[x_0:x_1:\dots:x_n]$ be homogeneous coordinates of $\PP(V) \cong \PP^n$.
	Throughout the paper, by a slight abuse of notation, we will occasionally identify a quadric in $\PP^n$ with its defining equation, or with the symmetric matrix representing it in the chosen coordinates $[x_0:x_1:\dots:x_n]$.
	This identification should cause no confusion, as it will always be clear from the context.
	
	Denote by $Q_0,Q_1,\dots,Q_n$ a set of generators of a linear system $\mathcal{L}$ of quadrics in $\PP^n$ with $\dim \mathcal{L} \le n+1$, and let
	\[
	F:= \lambda_0 Q_0 + \dots + \lambda_n Q_n
	\]
	be a general element of $\mathcal{L}$.
	If we compute the gradient of $F$ (with respect to the partial derivatives $\frac{\partial}{\partial x_i}$, for $0 \le i \le n$), it is straightforward to note that the equation
	\begin{equation*}
		\nabla F = 0
	\end{equation*}
	can be regarded as a linear system in the $\lambda_k$, for $0 \le k \le n$, variables.
	Denote by $\mathbf{W}_{\mathcal{L}}$ the matrix associated to such linear system.
	This matrix can be expressed as
	\begin{equation}\label{eq_weddle_matrix_explicit}
		\begin{bmatrix}
			\frac{\partial Q_0}{\partial x_0} & \dots & \frac{\partial Q_n}{\partial x_0} \\
			\vdots & \ddots & \vdots \\
			\frac{\partial Q_0}{\partial x_n} & \dots & \frac{\partial Q_n}{\partial x_n}
		\end{bmatrix}
	\end{equation}
	see also Remark~\ref{prop_weddle_definition_alternative} below.
	Of course, the entries of $\mathbf{W}_{\mathcal{L}}$ are linear forms in $x_0,\dots,x_n$.
	
	\begin{definition}
		The \emph{Weddle locus} of $\mathcal{L}$ is
		\[
		W(\mathcal{L}) := \left\{ [P_0:\dots:P_n] \in \PP^n \ \text{s.t.} \ \det \mathbf{W}_{\mathcal{L}}\big|_{[P_0:\dots:P_n]} = 0 \right\}
		\]
		and $\mathbf{W}_{\mathcal{L}}$ is called the \emph{Weddle matrix} of $\mathcal{L}$.
	\end{definition}
	
	\begin{remark}\label{rem_geometric_interpretation_weddle}
		If $\dim \mathcal{L} = n+1$, then the Weddle locus of $\mathcal{L}$ has a relevant geometric meaning.
		Namely, a point $P$ belongs to the Weddle locus of $\mathcal{L}$ if and only if $P$ is singular for some quadric in the linear system $\mathcal{L}$.
		
		Indeed, if $P$ is singular for some $F= \lambda_0 Q_0 + \dots + \lambda_n Q_n$, then
		\[
		\nabla F (P) = \mathbf{W}_{\mathcal{L}}\big|_{P} \begin{pmatrix}
			\lambda_0 \\
			\vdots \\
			\lambda_n
		\end{pmatrix}
		\]
		equals zero, where the symbol $\nabla F (P)$ stands for the gradient of $F$ evaluated at $P$.
		Hence, $\mathbf{W}_{\mathcal{L}}\big|_{P}$ is singular.
		The converse follows by the explicit expression~\eqref{eq_weddle_matrix_explicit} of the matrix $\mathbf{W}_{\mathcal{L}}$.
		
		Note also that if $\dim \mathcal{L} < n+1$, then the Weddle locus $W(\mathcal{L})$ covers $\PP^n$, as the determinant of $\mathbf{W}_{\mathcal{L}}$ is identically zero.
	\end{remark}
	Therefore, the Weddle locus is independent of the choice of generators of $\mathcal{L}$, and if a transformation in $\operatorname{\PP GL}(V)$ acts on the generators, then $W(\mathcal{L})$ is transformed accordingly.
	\begin{remark}
		Suppose that $F$ is a cubic homogeneous polynomial.
		One can consider the linear system $\mathcal{L}_F$ generated by the quadrics defined by the partial derivatives of $F$.
		In this case, it follows by definition that the Weddle matrix $\mathbf{W}_{\mathcal{L}_F}$ is the Hessian matrix of $F$.
		Therefore, the Weddle locus of $\mathcal{L}_F$ is exactly the \emph{Hessian locus} of the cubic hypersurface defined by $F$.
		See \cite{BFP24}, and references therein, for further details concerning the Hessian loci.
		
		Note that, in view of Remark~\ref{rem_system_in_S+N}, $\mathcal{L}_F$ is an example of a linear system of quadrics represented by a symmetric tensor.
	\end{remark}
	
	In the following crucial remark, we present a more efficient method for computing the Weddle locus, namely through the cubic tensor associated to the linear system of quadrics $\mathcal{L}$.
	
	\begin{remark}\label{prop_weddle_definition_alternative}
		The Weddle locus $W(\mathcal{L})$ can be computed by means of the following steps:
		\begin{enumerate}
			\item fix $n+1$ generators $Q_0, \dots, Q_n$ of $\mathcal{L}$ as above;\label{stp1_prop_weddle_definition_alternative}
			\item consider the cubic tensor $T$ associated to $\mathcal{L}$, which is $T_{ijk} = (Q_k)_{ij}$;\label{stp2_prop_weddle_definition_alternative}
			\item define a new tensor $\mathbf{T}$ as $\mathbf{T}_{ijk} :=  x_j T_{ijk}$;\label{stp3_prop_weddle_definition_alternative}
			\item compute the contraction $\mathbf{T}^{\{2\}}$ of $\mathbf{T}$ along the second index;\label{stp4_prop_weddle_definition_alternative}
			\item compute the determinant of the resulting matrix $\mathbf{T}^{\{2\}}$.\label{stp5_prop_weddle_definition_alternative}
		\end{enumerate}
		Then, $W(\mathcal{L})$ is the zero locus of $\det \mathbf{T}^{\{2\}}$.
		
		To see this, denote by $M_F$ the matrix associated with the general element $F$ of $\mathcal{L}$.
		By the relation
		\[
		\nabla F = \nabla (\mathbf{x}^t  M_F  \mathbf{x}) = 2 \mathbf{x}^t M_F = 2 \sum_{k=0}^n \lambda_k \mathbf{x}^t Q_k
		\]
		we get
		\begin{align*}
			\mathbf{W}_{\mathcal{L}} &=
			\begin{bmatrix}
				\operatorname{coeff}_{\lambda_k}\frac{\partial F}{\partial x_i}
			\end{bmatrix}_{i,k} \\
			&=
			\begin{bmatrix}
				\frac{\partial (\mathbf{x}^t  Q_k  \mathbf{x})}{\partial x_i}
			\end{bmatrix}_{i,k} \\
			&=
			2 \begin{bmatrix}
				 \mathbf{x}^t \cdot \operatorname{Col}_{i+1}(Q_k)
			\end{bmatrix}_{i,k} \\
			&=
			2 \begin{bmatrix}
				\sum_{j} x_j (Q_k)_{ij}
			\end{bmatrix}_{i,k} \\
			&=
			2 \begin{bmatrix}
				\sum_{j} x_j T_{ijk}
			\end{bmatrix}_{i,k} \\
			&= 2 \mathbf{T}^{\{2\}}
		\end{align*}
		and we finish by computing the determinant.
	\end{remark}
	\begin{remark}
		It is also possible to rephrase Remark~\ref{prop_weddle_definition_alternative} as follows.
		Substitute $\mathbf{T}$ with the tensor $\tilde{\mathbf{T}}$ defined by $\tilde{\mathbf{T}}_{ijk} :=  x_i T_{ijk}$.
		Perform the contraction $\tilde{\mathbf{T}}^{\{1\}}$ of $\tilde{\mathbf{T}}$ with respect to the first index.
		Define the Weddle locus of $\mathcal{L}$ as the zero locus of the determinant of $\tilde{\mathbf{T}}^{\{1\}}$.
		This definition is equivalent to Remark~\ref{prop_weddle_definition_alternative} as the tensor $T$ associated to $\mathcal{L}$ is partially symmetric with respect to the first two indices.
	\end{remark}
	
	The next result highlights how properties of the linear system of quadrics are related to the geometry of its Weddle locus.
	\begin{theorem}\label{prop_base_point_is_singular}
		Every base point of $\mathcal{L}$ is a singular point of $W(\mathcal{L})$.
	\end{theorem}
	\begin{proof}
		Suppose that $\mathcal{L}$ is generated by $Q_0,\dots,Q_n$ and let $P$ a base point of $\mathcal{L}$.
		Without loss of generality, assume that $P = [1,0,\dots,0]$.
		By Euler's Theorem we have that
		\[
		\sum_{i=0}^n \bigg[ x_i \frac{\partial Q_k}{\partial x_i} \bigg]_{[x]=P} = 2 Q_k(P) = 0
		\]
		hence, for $0 \le k \le n$, we immediately get
		\begin{equation}\label{eq_0_entries_base_point}
		\frac{\partial Q_0}{\partial x_0}(P) = \dots = \frac{\partial Q_n}{\partial x_0}(P) = 0 .
		\end{equation}
		From the chain of equalities
		\begin{equation*}
			\frac{\partial Q_k}{\partial x_i}(P) = 2 \sum_{j=0}^{n} \bigg[ x_j (Q_k)_{ij} \bigg]_{[x]=P} = 2 (Q_k)_{i0}
		\end{equation*}
		Equation~\eqref{eq_0_entries_base_point} reads as
		\begin{equation*}
			(Q_0)_{00} = \dots = (Q_n)_{00} = 0
		\end{equation*}
		for the tensor associated to $\mathcal{L}$.
		We follow here the notation of Remark~\ref{prop_weddle_definition_alternative}.
		
		By the Laplace's expansion of the matrix $\mathbf{W}_{\mathcal{L}}$ with respect to its first row
		\begin{equation*}
			\det \mathbf{W}_{\mathcal{L}} = \det \begin{bmatrix}
				\frac{\partial Q_0}{\partial x_0} & \dots & \frac{\partial Q_n}{\partial x_0} \\
				\vdots & \ddots & \vdots \\
				\frac{\partial Q_0}{\partial x_n} & \dots & \frac{\partial Q_n}{\partial x_n}
			\end{bmatrix}
			= \sum_{k = 0}^{n} (-1)^k \frac{\partial Q_k}{\partial x_0} \cdot {W}_{0k}
		\end{equation*}
		we immediately see, by Equation~\eqref{eq_0_entries_base_point}, that $P \in W(\mathcal{L})$.
		Here, ${W}_{ij}$ denotes the determinant of the $n \times n$ sub-matrix obtained by removing from $\mathbf{W}_{\mathcal{L}}$ the $i$-th row and $j$-th column.
		
		We need to show that $\frac{\partial \det\mathbf{W}_{\mathcal{L}}}{\partial x_s}\big|_P = 0$, for $0 \le s \le n$.
		By differentiating, we get
		\begin{align*}
			\frac{\partial}{\partial x_s} \det \mathbf{W}_{\mathcal{L}} (P) &= \sum_{k = 0}^{n} (-1)^k \bigg[ \frac{\partial^2 Q_k}{\partial x_s \partial x_0} \cdot {W}_{0k} + \frac{\partial Q_k}{\partial x_0} \cdot \frac{\partial {W}_{0k}}{\partial x_s} \bigg]_{[x]=P} \\
			&= \sum_{k = 0}^{n} (-1)^k \frac{\partial^2 Q_k}{\partial x_s \partial x_0}(P) \cdot {W}_{0k}(P)
		\end{align*}
		where the last equality follows again from Equation~\eqref{eq_0_entries_base_point}.
		The computation
		\begin{align*}
			\frac{\partial^2 Q_k}{\partial x_s \partial x_i} &= \frac{\partial}{\partial x_s} \left( 2 \sum_{j=0}^{n} x_j (Q_k)_{ij} \right) \\
			&= 2 \sum_{j=0}^{n} \delta_{sj}(Q_k)_{ij} + x_j \frac{\partial (Q_k)_{ij}}{\partial x_s} \\
			&= 2(Q_k)_{is} + 2\sum_{j=0}^{n} x_j \frac{\partial (Q_k)_{ij}}{\partial x_s}
		\end{align*}
		yields
		\[
		\frac{\partial^2 Q_k}{\partial x_s \partial x_0}(P) = 2(Q_k)_{0s} + 2\frac{\partial (Q_k)_{00}}{\partial x_s} = 2(Q_k)_{0s}
		\]
		Also, observe that, in the end, the Weddle matrix evaluated in $P$ reads
		\[
		\mathbf{W}_{\mathcal{L}} \big|_P = 2\begin{bmatrix}
			(Q_0)_{00} & \dots & (Q_n)_{00} \\
			(Q_0)_{10} & \dots & (Q_n)_{10} \\
			\vdots & \ddots & \vdots \\
			(Q_0)_{n0} & \dots & (Q_n)_{n0}
		\end{bmatrix} = 2\begin{bmatrix}
		0 & \dots & 0 \\
		(Q_0)_{01} & \dots & (Q_n)_{01} \\
		\vdots & \ddots & \vdots \\
		(Q_0)_{0n} & \dots & (Q_n)_{0n}
		\end{bmatrix}
		\]
		To conclude, we have to show that
		\begin{equation*}\label{eq_P_singular_weddle}
			\frac{\partial}{\partial x_s} \det \mathbf{W}_{\mathcal{L}}(P) = 2\sum_{k = 0}^{n} (-1)^k (Q_k)_{0s} \cdot {W}_{0k}(P)
		\end{equation*}
		vanishes for $0 \le s \le n$.
		
		For $s=0$, note that
		\[
		\frac{1}{2}\frac{\partial}{\partial x_0}\det \mathbf{W}_{\mathcal{L}}(P) =
		(Q_0)_{00} \cdot  {W}_{00}(P) - (Q_1)_{00} \cdot {W}_{01}(P) + \dots + (-1)^n (Q_n)_{00} \cdot {W}_{0n}(P)
		\]
		is trivially zero, as it is the Laplace's expansion of the matrix $\frac{1}{2} \mathbf{W}_{\mathcal{L}} \big|_P$ with respect to its first row, the latter consisting only of zeros.
		
		For $s \neq 0$, we observe that the expression
		\[
		\frac{1}{2}\frac{\partial}{\partial x_s}\det \mathbf{W}_{\mathcal{L}}(P) =
		(Q_0)_{0s} \cdot  {W}_{00}(P) - (Q_1)_{0s} \cdot {W}_{01}(P) + \dots + (-1)^n (Q_n)_{0s} \cdot {W}_{0n}(P)
		\]
		equals exactly the Laplace's expansion of the matrix
		\[
		\mathbf{W}_s := \begin{bmatrix}
			(Q_0)_{0s} & \dots & (Q_n)_{0s} \\
			(Q_0)_{01} & \dots & (Q_n)_{01} \\
			\vdots & \ddots & \vdots \\
			(Q_0)_{0n} & \dots & (Q_n)_{0n}
		\end{bmatrix}
		\]
		with respect to its first row.
		The matrix $\mathbf{W}_s$ has been obtained from $\frac{1}{2} \mathbf{W}_{\mathcal{L}} \big|_P$ by replacing the first row of the latter with its $s$-th row.
		Given that the first and $s$-th row of $\mathbf{W}_s$ are equal, we can conclude that $\frac{\partial \det\mathbf{W}_{\mathcal{L}}}{\partial x_s}\big|_P$ equals zero.
	\end{proof}
	
	The converse of Theorem~\ref{prop_base_point_is_singular} does not hold, as the following example shows.
	\begin{example}\label{ex_bpf_linear_sys_singular_weddle}
		Let $\mathcal{L}$ be the linear system of $\PP^2$ generated by $x_0^2, x_1^2, x_2^2$.
		By following the steps of Remark~\ref{prop_weddle_definition_alternative}, we immediately compute the contraction matrix $\mathbf{T}^{\{2\}}$, which is
		\[
		\begin{bmatrix}
			x_0 & 0 & 0 \\
			0 & x_1 & 0 \\
			0 & 0 & x_2
		\end{bmatrix}
		\]
		hence the Weddle locus of $\mathcal{L}$ is defined by the equation $ x_0 x_1 x_2 = 0$.
		It is clear that $W(\mathcal{L})$ has three singular points, though $\mathcal{L}$ is base point free.
	\end{example}

\section{Rank and Weddle loci}
	In the next series of results we relate the {rank} of a linear system $\mathcal{L}$ with the geometry of its Weddle locus.
	
	Suppose that $\dim V = n+1$.
	Recall, from Remark~\ref{rem_geometric_interpretation_weddle}, that if $\dim \mathcal{L} \le n$ then $W(\mathcal{L})$ equals $\PP^n$.
	Therefore, from now on assume that $\dim \mathcal{L} = n+1$.
	Let
	\begin{equation*}
		v_{n,2} \colon \PP(V) \cong \PP^n \to \PP({\operatorname{S}}^2 V) \cong \PP^{N}
	\end{equation*}
	be the degree~$2$ Veronese map of $\PP^n$, where $N = \frac{n(n+3)}{2}$.
	Set $X := v_{n,2}(\PP^n)$, and let $S$ denote the Segre embedding of $\PP^n \times X$ into $\PP^{(n+1)(N+1)-1}$.
	Recall that the (symmetric) rank of $\mathcal{L}$, which depends on the Grassmann-secant varieties that contain the linear space in $\PP({\operatorname{S}}^2 V)$ identified by $\mathcal{L}$, is equal to the $S$-rank of the partially symmetric tensor $T$ associated to $\mathcal{L}$.
	See, for instance, \cite{Lan12, BBCC13}.
	Indeed, there is a one-to-one correspondence between the decompositions of $\mathcal{L}$ and the decompositions of $T$.
	
	The rank of the linear system $\mathcal{L}$ is, trivially, bounded by
	\begin{equation}\label{eq_chain_dimensions_rank}
		n+1 = \dim V \le \ \ \operatorname{rk} \mathcal{L} \ \ \le \dim {\operatorname{S}}^2 V = \frac{(n+1)(n+2)}{2} .
	\end{equation}
	
	Since in the case $n=1$ there is nothing to say, we start with the following.
	
\subsection{Case $n = 2$}\label{subsect_n=2}
	Here $\dim \mathcal{L} = 3$.
	Therefore, $\mathcal{L}$ identifies a projective plane in $\PP({\operatorname{S}}^2 V) \cong \PP^5$, and the chain of inequalities~\eqref{eq_chain_dimensions_rank} becomes
	\begin{equation*}
		3 \le \operatorname{rk} \mathcal{L} \le 6 .
	\end{equation*}
We will see below that the rank of the tensor, for a generic choice of $\mathcal L$, is $4$.

	We start with the following useful observation.
	\begin{remark}\label{rem_weddle_matrix_conics}
		Suppose that the linear system $\mathcal{L}$ is generated by the conics
		\begin{equation*}
			\begin{bmatrix}
				A_{00} & A_{01} & A_{02} \\
				A_{01} & A_{11} & A_{12} \\
				A_{02} & A_{12} & A_{22}
			\end{bmatrix}, \ \
			\begin{bmatrix}
				B_{00} & B_{01} & B_{02} \\
				B_{01} & B_{11} & B_{12} \\
				B_{02} & B_{12} & B_{22}
			\end{bmatrix}, \ \
			\begin{bmatrix}
				C_{00} & C_{01} & C_{02} \\
				C_{01} & C_{11} & C_{12} \\
				C_{02} & C_{12} & C_{22}
			\end{bmatrix} .
		\end{equation*}
		By applying the steps~\eqref{stp1_prop_weddle_definition_alternative}-\eqref{stp4_prop_weddle_definition_alternative} of Remark~\ref{prop_weddle_definition_alternative}, we have that the explicit form of the Weddle matrix~$\frac{1}{2}\mathbf{W}_{\mathcal{L}}$ associated to $\mathcal{L}$ is
		\begin{equation*}
			\begin{bmatrix}
				A_{00}x_0 + A_{01}x_1 + A_{02}x_2 & B_{00}x_0 + B_{01}x_1 + B_{02}x_2 & C_{00}x_0 + C_{01}x_1 + C_{02}x_2\\
				A_{01}x_0 + A_{11}x_1 + A_{12}x_2 & B_{01}x_0 + B_{11}x_1 + B_{12}x_2 & C_{01}x_0 + C_{11}x_1 + C_{12}x_2\\
				A_{02}x_0 + A_{12}x_1 + A_{22}x_2 & B_{02}x_0 + B_{12}x_1 + B_{22}x_2 & C_{02}x_0 + C_{12}x_1 + C_{22}x_2
			\end{bmatrix} .
		\end{equation*}
	\end{remark}
	
	\begin{example}
		We already mentioned in Remark~\ref{rem_geometric_interpretation_weddle} that if $\dim \mathcal{L} < n+1$ (here $n=2$), then $W(\mathcal{L}) = \PP^n$.
		This is not the only case in which the Weddle locus is trivial.
		For instance, consider the linear system $\mathcal{L}$ generated by the following (degenerate) conics
		\begin{equation*}
			\begin{bmatrix}
				a & a & b \\
				a & a & b \\
				b & b & c
			\end{bmatrix}, \ \
			\begin{bmatrix}
				d & d & e \\
				d & d & e \\
				e& e & f
			\end{bmatrix}, \ \
			\begin{bmatrix}
				g & g & h \\
				g & g & h \\
				h & h & i
			\end{bmatrix} .
		\end{equation*}
		Following Remark~\ref{rem_weddle_matrix_conics}, we see that the Weddle matrix associated to $\mathcal{L}$ is
		\begin{equation*}
			\begin{bmatrix}
				ax_0 + ax_1 + bx_2 & dx_0 + dx_1 + ex_2 & gx_0 + gx_1 + hx_2\\
				ax_0 + ax_1 + bx_2 & dx_0 + dx_1 + ex_2 & gx_0 + gx_1 + hx_2\\
				bx_0 + bx_1 + cx_2 & ex_0 + ex_1 + fx_2 & hx_0 + hx_1 + ix_2
			\end{bmatrix} .
		\end{equation*}
		As this matrix has two equal rows, it is clear the Weddle locus of $\mathcal{L}$ is trivial, even if $\dim \mathcal{L} = 3 $.
		This means that each point in $\PP^2$ is singular for some conic in $\mathcal{L}$.
	\end{example}
	
	Now, we investigate the cases $\operatorname{rk} \mathcal{L} = 3, 4$.
	
\subsubsection*{Rank $3$ case}
	Let $\mathcal{L}$ be a general linear system of rank~$3$.
	Hence, there is a projective plane $\mathcal{L}^\prime \cong \PP^2$ containing $\mathcal{L}$, which intersects the Veronese variety $v_{2,2}(\PP^2) \subset \PP^5$ in three points.
	Obviously, by dimensional reasons, $\mathcal{L}^\prime = \mathcal{L}$.
	\begin{proposition}\label{prop_dim3_rk3}
		The Weddle locus of a general linear system $\mathcal{L}$ of rank~$3$ splits in the union of~$3$ lines.
	\end{proposition}
	\begin{proof}
		By a change of coordinates, it is always possible to assume that the three intersection points of $\mathcal{L}^\prime$ and $v_{2,2}(\PP^2)$ mentioned above are $x_0^2, x_1^2, x_2^2$. Indeed the three points of intersection cannot be collinear, since $v_{2,2}(\PP^2)$ is cut by quadrics and contains no lines.
		In this way, we are exactly in the same situation of Example~\ref{ex_bpf_linear_sys_singular_weddle}.
	\end{proof}
	\begin{remark}\label{rem_converse_rank3_3lines}
		Observe that the converse of the previous proposition does not hold.
		Indeed, if $\mathcal{L}$ is a linear system with three, non collinear, base points, then $\mathcal{L}$ does not contain double lines.
		Hence, the projective plane associated to $\mathcal{L}$ does not intersect the Veronese variety $v_{2,2}(\PP^2)$, therefore $\mathcal{L}$ is not of rank~$3$.
		On the other hand, by Theorem~\ref{prop_base_point_is_singular}, $W(\mathcal{L})$ is a cubic with three, non collinear, singular points, thus $W(\mathcal{L})$ splits in the union of three lines.
	\end{remark}
	
\subsubsection*{Rank $4$ case}
	Since the Veronese variety $v_{2,2}(\PP^2)$ is a quartic surface in $\PP^5$, a general projective space $\mathcal{L}^\prime \cong \PP^3$ intersects $v_{2,2}(\PP^2)$ in four points.
	Therefore, a general linear system $\mathcal{L}$ (of dimension~$3$) which is contained in $\mathcal{L}^\prime$ has rank $\operatorname{rk} \mathcal{L} \le 4$.
So, the rank of the associated tensor is $4$, by \cite{BBCC13}.
	
	Observe that, by definition, the Weddle locus of a linear system of quadrics in $\PP^n$ is a \emph{determinantal} hypersurface of degree $n+1$ (or is $\PP^n$).
	Here $n=2$, thus $W(\mathcal{L})$ is a determinantal plane cubic.
	Moreover, recall that by \cite{Bea00} the general plane cubic is determinantal.
	The first main result of this section is the following.
	
	\begin{theorem}
		A general elliptic cubic curve in $\PP^2$ is the Weddle locus of a suitable linear system $\mathcal{L}$ of rank~$4$.
	\end{theorem}
	\begin{proof}
		Recall that given a smooth plane cubic curve $\mathcal{C}$ whose Weierstrass normal form is
		\begin{equation*}
			x_0 x_2^2 = x_1^3 + a x_0^2 x_1 + b x_0^3
		\end{equation*}
		the $j$-invariant of $\mathcal{C}$ is the quantity
		\begin{equation*}
			j(\mathcal{C}) = 256  \frac{27 a^3}{4 a^3 + 27 b^2} .
		\end{equation*}
		We will use the well known fact that two smooth plane cubic curves are projectively equivalent if and only if they have the same $j$-invariant.
		
		Recall that, by Remark~\ref{rem_system_in_S+N}, we can consider the map
		\[
		w \colon {\operatorname{S}}^3 V \oplus {\operatorname{N}_1} V \to \PP({\operatorname{S}}^3 V) \cong \PP^9
		\]
		which sends a partially symmetric tensor $T$ to the Weddle locus $W(\mathcal{L}_T)$ of its associated linear system of conics $\mathcal{L}_T$.
		By composing $w$ with the $j$-invariant, we can consider the diagram
		\begin{equation*}
			\begin{tikzcd}
				{\operatorname{S}}^3 V \oplus {\operatorname{N}_1} V \arrow[dr,"J",',dashed] \arrow[rr,"w"]{}
				& & \PP({\operatorname{S}}^3 V) \arrow[dl,"j",dashed] \\
				& \C
			\end{tikzcd}
		\end{equation*}
		which is defined on the partially symmetric tensors which have a smooth Weddle locus.
		Anyhow, $J$ is either dominant on $\C$ or its image is the singleton $\{ \alpha \}$, for some $\alpha \in \C$.
		We prove that the former case holds, by constructing two tensors $T_1$ and $T_2$ in ${\operatorname{S}}^3 V \oplus {\operatorname{N}_1} V$, whose Weddle loci have different $j$-invariants.
		
		Let $T_1$ be the tensor associated to the linear system generated by the conics
		\begin{equation*}
			x_0^2 + 2 x_0 x_1 + 2 x_0 x_2 + 2 x_1 x_2, \quad
			x_0^2 + x_1^2, \quad
			x_0^2 + x_2^2 .
		\end{equation*}
		By Remark~\ref{rem_weddle_matrix_conics} its (normalized) Weddle matrix is
		\begin{equation*}
			\begin{bmatrix}
				x_0 + x_1 + x_2 & x_0  & x_0    \\
				x_0 + x_2       & x_1  & 0  	\\
				x_0 + x_1       &  0   & x_2
			\end{bmatrix}
		\end{equation*}
		and its Weddle locus is the smooth cubic
		\[
		\mathcal{C}_1 := \{ x_0 x_1 x_2 - x_0^2 x_1 - x_0 x_1^2 - x_0^2 x_2 - x_0 x_2^2 + x_1^2 x_2 + x_1 x_2^2 = 0 \}
		\]
		whose Weierstrass form in the affine chart $\{x_0 \neq 0\}$ is
		\[
		x_2^2 = x_1^3 - \frac{121}{48}x_1 + \frac{845}{864}
		\]
		therefore, $j(\mathcal{C}_1) = \frac{1771561}{612} \sim 2894,71$.
		
		Let $T_2$ be the tensor associated to the linear system generated by the conics
		\begin{equation*}
			x_0^2 + 2 x_0 x_1 + 2 x_0 x_2 + x_1^2 + x_2^2 , \quad
			2x_0 x_1 - x_1^2, \quad
			2x_0 x_2 - x_2^2 .
		\end{equation*}
		Again by Remark~\ref{rem_weddle_matrix_conics} its Weddle matrix is
		\begin{equation*}
			\begin{bmatrix}
				x_0 + x_1 + x_2 & x_1        & x_2      \\
				x_0 + x_1       & x_0 - x_1  & 0		\\
				x_0 + x_2       &  0         & x_0 - x_2
			\end{bmatrix}
		\end{equation*}
		and its Weddle locus is
		\[
		\mathcal{C}_2 := \{ x_0 x_1 x_2 - x_0^2 x_1 - 2 x_0 x_1^2 - x_0^2 x_2 - 2 x_0 x_2^2 + x_0^3 + 2 x_1^2 x_2 + 2 x_1 x_2^2 = 0 \}
		\]
		whose Weierstrass form in the affine chart $\{x_0 \neq 0\}$ is
		\[
		x_2^2 = x_1^3 - \frac{1633}{48}x_1 + \frac{61201}{864}
		\]
		hence, $j(\mathcal{C}_2) = \frac{4354703137}{352512} \sim 12.353,35$.
	\end{proof}
	
\subsection{Case $n = 3$}\label{subsect_n=3}
	In this case $\dim \mathcal{L} = 4$.
	Therefore, $\mathcal{L}$ identifies a projective~$3$-space in $\PP({\operatorname{S}}^2 V) \cong \PP^9$, and the inequalities in~\eqref{eq_chain_dimensions_rank} are
	\begin{equation*}
		4 \le \operatorname{rk} \mathcal{L} \le 10 .
	\end{equation*}
	We investigate the cases $\operatorname{rk} \mathcal{L} = 4, 5$.
	\begin{remark}
		By the Second Terracini's Lemma (see \cite{CC01}, \cite{BBCC13}), the rank $r$ of a generic~$4$-dimensional linear system of quadrics in $\PP^3$ is equal to the minimal value for which the $r$-th secant variety of $\PP^3 \times X$ fills the ambient space $\PP^{39}$, where $X$ is the degree $2$ Veronese embedding of $\PP^3$.
		One computes immediately that the expected value for the generic rank is $6$. It follows from \cite{AB09} that $\PP^3 \times X$, which is a $(1,2)$ Segre-Veronese variety, is not defective. Hence the generic rank equals the expected value.
	\end{remark}

We focus on the cases where the rank of the tensor is smaller than $6$, the generic rank.

\subsubsection*{Rank $4$ case}
	Let $\mathcal{L}$ be a general linear system of rank~$4$.
	In this case, there is a $\mathcal{L}^\prime \cong \PP^3$ which contains $\mathcal{L}$ and intersects the Veronese variety $v_{3,2}(\PP^3) \subset \PP^9$ in (at least) four independent points.
	Obviously, by dimensional reasons, $\mathcal{L}^\prime = \mathcal{L}$.
	\begin{proposition}\label{prop_dim4_rk4}
		The Weddle locus of a general linear system $\mathcal{L}$ of rank~$4$ splits in the union of~$4$ planes.
	\end{proposition}
	\begin{proof}
		The proof is analogous to that of Proposition~\ref{prop_dim3_rk3}.
		Indeed, by a change of coordinates, it is always possible to assume that the four intersection points are $x_0^2, x_1^2, x_2^2, x_3^2$.
		In this way, we proceed similarly as in Example~\ref{ex_bpf_linear_sys_singular_weddle} since the Weddle matrix of $\mathcal{L}$ is
		\[
		\begin{bmatrix}
			x_0 & 0 & 0 & 0\\
			0 & x_1 & 0 & 0\\
			0 & 0 & x_2 & 0\\
			0 & 0 & 0 & x_3
		\end{bmatrix}
		\]
		hence $W(\mathcal{L})$ is defined by the equation $ x_0 x_1 x_2 x_3 = 0$.
	\end{proof}
	
	\begin{remark}
		As in the case of $\PP^2$, see Remark~\ref{rem_converse_rank3_3lines}, there are~$4$-dimensional linear systems of quadrics in $\PP^3$ with rank greater than~$4$ whose Weddle locus splits in the union of~$4$ planes.
	\end{remark}

\subsubsection*{Rank $5$ case}
	Let $\mathcal{L}$ be a general linear system of rank~$5$.
	By definition, there exists a projective space $\mathcal{L}^\prime \cong \PP^4$ which contains $\mathcal{L}$ and meets the Veronese variety $v_{3,2}(\PP^3)$ in (at least) five independent points.
	We prove the second main result of this section.
	
	\begin{theorem}\label{thm_dim4_rk5}
		Let $\mathcal{L}$ be a general linear system of rank~$5$.
		{Then, the Weddle locus of $\mathcal{L}$ has~$10$ singular points.}
	\end{theorem}
	\begin{proof}
		Given that the linear system $\mathcal{L}$ is general, the five points mentioned above are in general position.
		Thus, after changing coordinates, we can assume that they are
		\[
		x_0^2, \ x_1^2, \ x_2^2, \ x_3^2, \ (x_0 + x_1 + x_2 + x_3)^2 .
		\]
		
		\textbf{Step 1.}
		First, we explicitly compute the Weddle matrix $\mathbf{W}_{\mathcal{L}}$.
		Suppose that, for $0 \le k \le 3$, $\mathcal{L}$ is generated by
		\begin{equation*}
			Q_k = \alpha_k x_0^2 + \beta_k x_1^2 + \gamma_k x_2^2 + \delta_k x_3^2 + \epsilon_k (x_0 + x_1 + x_2 + x_3)^2
		\end{equation*}
		whose associated matrix, by assuming that $\epsilon_k = 1$, are
		\begin{equation*}
			\begin{bmatrix}
				1 + \alpha_0 	& 1        	& 1       	& 1 \\
				1       	& 1 + \beta_0 & 1		 	& 1 \\
				1       	&  1  		& 1+ \gamma_0	& 1 \\
				1       	&  1        & 1 		& 1 + \delta_0
			\end{bmatrix},
			\dots,
			\begin{bmatrix}
				1 + \alpha_3 	& 1        	& 1       	& 1 \\
				1       	& 1 + \beta_3 & 1		 	& 1 \\
				1       	&  1  		& 1+ \gamma_3	& 1 \\
				1       	&  1        & 1 		& 1 + \delta_3
			\end{bmatrix} .
		\end{equation*}
		By adapting Remark~\ref{rem_weddle_matrix_conics} to dimension~$3$, we can immediately deduce the explicit form of the Weddle matrix $\mathbf{W}_{\mathcal{L}}$ of $\mathcal{L}$ as
		\begin{equation*}
			2 \begin{bmatrix}
				(1+\alpha_0)x_0 + x_1 + x_2 + x_3 & \dots & (1+\alpha_3)x_0 + x_1 + x_2 + x_3\\
				x_0 + (1+\beta_0)x_1 + x_2 + x_3  & \dots & x_0 + (1+\beta_3)x_1 + x_2 + x_3\\
				x_0 + x_1 + (1+\gamma_0)x_2 + x_3 & \dots & x_0 + x_1 + (1+\gamma_3)x_2 + x_3\\
				x_0 + x_1 + x_2 + (1+\delta_0)x_3 & \dots & x_0 + x_1 + x_2 + (1+\delta_3)x_3
			\end{bmatrix} .
		\end{equation*}
		Let $M = (M_{ij})$ be the matrix
		\begin{equation*}
			\begin{bmatrix}
			\alpha_0 & \alpha_1 & \alpha_2 & \alpha_3\\
			\beta_0 & \beta_1 & \beta_2 & \beta_3\\
			\gamma_0 & \gamma_1 & \gamma_2 & \gamma_3\\
			\delta_0 & \delta_1 & \delta_2 & \delta_3\\
			\end{bmatrix} .
		\end{equation*}
		Note that $M$ is a general matrix, as $\mathcal{L}$ is a general linear system.
		Therefore, for $0 \le i,j \le 3$, we can write
		\begin{equation*}
			\frac{1}{2} (\mathbf{W}_{\mathcal{L}})_{ij} = (x_0 + x_1 + x_2 + x_3) + x_i M_{ij}
		\end{equation*}
		and, more concisely, if we set $\xi := x_0 + x_1 + x_2 + x_3$ and
		\begin{equation*}
			\Xi := \begin{bmatrix}
				\xi & \xi & \xi & \xi \\
				\xi & \xi & \xi & \xi \\
				\xi & \xi & \xi & \xi \\
				\xi & \xi & \xi & \xi \\
			\end{bmatrix}, \ 
			D := \begin{bmatrix}
				x_0 & 0 & 0 & 0\\
				0 & x_1 & 0 & 0\\
				0 & 0 & x_2 & 0\\
				0 & 0 & 0 & x_3\\
			\end{bmatrix}
		\end{equation*}
		we write
		\begin{equation*}
			\frac{1}{2} \mathbf{W}_{\mathcal{L}} = \Xi + D M .
		\end{equation*}
		
		\textbf{Step 2.}
		Given two $m \times m$ matrices $A$ and $B$, recall that the determinant of $A+B$ can be expressed as the sum of the determinants of $2^m$ matrices, obtained by replacing, for each subset of columns, the corresponding columns of $A$ with those of $B$.
		
		Since we want to compute $\det (\Xi + D M)$, and all columns of $\Xi$ are identical, only five determinants contribute.
		These are: the determinant of $D M$ itself, and the four determinants of the matrices obtained by replacing, for each $1 \le \ell \le 4$, the $\ell$-th column of $D M$ with the corresponding column of $\Xi$.
		From this, it is easy to deduce that the determinant of $\frac{1}{2} \mathbf{W}_{\mathcal{L}}$ can be expressed as
		\[
		F := x_0 x_1 x_2 L_3 + x_0 x_1 L_2 x_3 + x_0 L_1 x_2 x_3 + L_0 x_1 x_2 x_3
		\]
		for suitable linear forms $L_0, L_1, L_2, L_3$.
		In what follows, we study the singularities of the Weddle locus of $\mathcal{L}$, which is $\{ F = 0\}$.
		
		\textbf{Step 3.}
		For $0\le s \le 3$, it is useful to write
		\[
		L_s = L_s^0 x_0 + L_s^1 x_1 + L_s^2 x_2 + L_s^3 x_3 ,
		\]
		where each $L_s^i \in \C$ depends on the entries of $M$.
		By expanding $F$, we can write
		\begin{equation}\label{eq_weddle_rank_5_expanded}
			\begin{aligned}
			F &= \left( \sum_{s = 0}^3 L_s^s \right) x_0 x_1 x_2 x_3 \\
			&+ \sum_{0 \le i < j < k \le 3} \sum_{ \substack{0 \le s \le 3 \\ s \notin \{i,j,k\}}} L_s^i x_i^2 x_j x_k + L_s^j x_i x_j^2 x_k + L_s^k x_i x_j x_k^2 .
		\end{aligned}
		\end{equation}
		
		Now, for each $0 \le s \le 3$, denote by $\mu_s$ the determinant of the matrix obtained by replacing the $(s+1)$-th row of $M$ with $(1,1,1,1)$.
		For instance,
		\begin{equation*}
			\mu_1 = \det \begin{bmatrix}
				\alpha_0 & \alpha_1 & \alpha_2 & \alpha_3\\
				1 & 1 & 1 & 1\\
				\gamma_0 & \gamma_1 & \gamma_2 & \gamma_3\\
				\delta_0 & \delta_1 & \delta_2 & \delta_3\\
			\end{bmatrix} .
		\end{equation*}
		
		Observe that the coefficients of $\det (\Xi + D M)$ are homogeneous polynomials in the entries of $M$.
		By implementing some computations in \texttt{CoCalc} (see \cite{CoCalc}), where we compared each monomial in~\eqref{eq_weddle_rank_5_expanded} with the corresponding monomial of $\det (\Xi + D M)$, we deduce the following equalities
		\[
		\sum_{s = 0}^3 L_s^s = \det M + \sum_{s = 0}^3 \mu_s
		\]
		and, for each $ 0 \le i < j < k \le 3$, $0 \le s \le 3$, $s \notin \{i,j,k\}$,
		\[
		L_s^i = L_s^j = L_s^k = \mu_s .
		\]
		Therefore, Equation~\eqref{eq_weddle_rank_5_expanded} becomes
		\begin{equation*}
			\begin{aligned}
				F &= \left(\det M + \sum_{s = 0}^3 \mu_s\right) x_0 x_1 x_2 x_3 \\
				&+ \sum_{0 \le i < j < k \le 3} \sum_{ \substack{0 \le s \le 3 \\ s \notin \{i,j,k\}}} \mu_s (x_i^2 x_j x_k + x_i x_j^2 x_k + x_i x_j x_k^2) .
			\end{aligned}
		\end{equation*}
		
		\textbf{Step 4.}
		Now, for each $0 \le \ell \le 3$, we differentiate $F$ with respect to $x_\ell$.
		We get
		\begin{equation*}
			\begin{aligned}
				\frac{\partial F}{\partial x_\ell} &= \left(\det M + \sum_{s = 0}^3 \mu_s\right) \prod_{i \neq \ell} x_i \\
				&+ \sum_{\substack{0 \le j < k \le 3 \\ \ell \notin \{j,k\}}} \sum_{ \substack{0 \le s \le 3 \\ s \notin \{j,k,\ell\}}} \mu_s (2x_\ell x_j x_k + x_j^2 x_k + x_j x_k^2) .
			\end{aligned}
		\end{equation*}
		From this explicit expression, we see that the ten points
		\begin{multline*}
			[1:0:0:0], [0:1:0:0], [0:0:1:0], [0:0:0:1],\\
			[1:-1:0:0], [1:0:-1:0], [1:0:0:-1],\\
			[0:1:-1:0], [0:1:0:-1], [0:0:1:-1]
		\end{multline*}
		annihilate the gradient of $F$.
		Thus, the Weddle locus of $\mathcal{L}$ has at least these ten singular points.
		
		\textbf{Step 5.}
		Finally, we observe that for a suitable choice of a general linear system of rank~$5$, equivalently, for a suitable choice of the matrix $M$, the corresponding Weddle locus has exactly the ten singular points listed above.
		For instance, one may take
		\[
		M = \begin{bmatrix}
			 1 & 0 & 1 & 1 \\
			 1 & 2 & 0 & 1 \\
			 0 & 1 & -1 & 1 \\
			 1 & 0 & 1 & 0 \\
		\end{bmatrix} .
		\]
		In this case, the quartic polynomial $F$ becomes
		\begin{align*}
		x_0 x_1 x_2 x_3 &- x_0^2 x_1 x_2 - x_0 x_1^2 x_2 - x_0 x_1 x_2^2 - x_0^2 x_1 x_3 - x_0 x_1^2 x_3 - x_0 x_1 x_3^2 \\
		&+ x_0^2 x_2 x_3 + x_0 x_2^2 x_3 + x_0 x_2 x_3^2 + x_1^2 x_2 x_3 + x_1 x_2^2 x_3 + x_1 x_2 x_3^2 .
		\end{align*}
		By using \texttt{Macaulay2} (see \cite{M2}) we computed the Hilbert polynomial of the ideal generated by the partial derivatives of $F$.
		In this way, we saw that the singular locus of the associated Weddle locus has projective dimension~$0$ and degree~$10$.
		This shows that the bound is optimal and the proof is over.
	\end{proof}
	
	The results established in this section allows us to derive information on the rank of a linear system of quadrics, starting from the geometry of its associated Weddle locus.
	\begin{corollary}\label{cor_sing_wed_less_10}
		Let $\mathcal{L}$ be a linear system of quadrics in $\PP^3$.
		{If the Weddle locus of $\mathcal{L}$ contains less than~$10$ singular points (\textit{e.g.}, if $W(\mathcal{L})$ is smooth), then the rank of $\mathcal{L}$ is~$\ge 6$.}
	\end{corollary}
	\begin{example}
		Consider the linear system in $\PP^3$, taken ``randomly", generated by the quadrics
		\begin{equation*}
			\resizebox{\textwidth}{!}{$
			\begin{bmatrix}
				0 & -\frac{3}{2} & 0 & 2 \\
				-\frac{3}{2} & 1 & -1 & \frac{1}{2} \\
				0 & -1 & -5 & 1 \\
				2 & \frac{1}{2} & 1 & 1 \\
			\end{bmatrix} ,
			\begin{bmatrix}
				1 & -\frac{1}{2} & -\frac{1}{2} & \frac{1}{2} \\
				-\frac{1}{2} & 0 & -\frac{1}{2} & 6 \\
				-\frac{1}{2} & -\frac{1}{2} & 1 & 4 \\
				\frac{1}{2} & 6 & 4 & 25 \\
			\end{bmatrix} ,
			\begin{bmatrix}
				-7 & 1 & -\frac{13}{2} & 0 \\
				1 & -1 & -\frac{3}{2} & \frac{1}{2} \\
				-\frac{13}{2} & -\frac{3}{2} & 101 & \frac{3}{2} \\
				0 & \frac{1}{2} & \frac{3}{2} & 4 \\
			\end{bmatrix} ,
			\begin{bmatrix}
				-4 & -\frac{9}{2} & \frac{13}{2} & 2 \\
				-\frac{9}{2} & 0 & 0 & -\frac{15}{2} \\
				\frac{13}{2} & 0 & -1 & -\frac{1}{2} \\
				2 & -\frac{15}{2} & -\frac{1}{2} & 1 \\
			\end{bmatrix}
			$}
		\end{equation*}
		In this case, by Remark~\ref{prop_weddle_definition_alternative}, the associated Weddle matrix is 
		\begin{equation*}
			\resizebox{\textwidth}{!}{$
				\begin{bmatrix}
				-\frac{3}{2} x_1 + 2 x_3 & x_0 -\frac{1}{2} x_1 -\frac{1}{2} x_2 + \frac{1}{2} x_3 & -7 x_0 + x_1  -\frac{13}{2} x_2 & -4 x_0 -\frac{9}{2} x_1 + \frac{13}{2} x_2 + 2 x_3\\
				-\frac{3}{2} x_0 + x_1 - x_2 + \frac{1}{2} x_3 & -\frac{1}{2} x_0  -\frac{1}{2} x_2 + 6 x_3 & x_0  -x_1  -\frac{3}{2}x_2 + \frac{1}{2}x_3 & -\frac{9}{2} x_0 -\frac{15}{2}x_3\\
				- x_1 -5 x_2 + x_3 & -\frac{1}{2} x_0 -\frac{1}{2} x_1 +x_2 +4 x_3 & -\frac{13}{2} x_0 -\frac{3}{2} x_1 +101 x_2 + \frac{3}{2}x_3 & \frac{13}{2} x_0  -x_2  -\frac{1}{2}x_3\\
				2 x_0 +\frac{1}{2} x_1 + x_2 + x_3 & \frac{1}{2} x_0 +6x_1 +4x_2 +25 x_3 & \frac{1}{2} x_1 + \frac{3}{2} x_2 + 4 x_3 & 2 x_0 -\frac{15}{2}x_1  -\frac{1}{2} x_2 + x_3
			\end{bmatrix}
			$}
		\end{equation*}
		and denote by $F$ its determinant.
		We used \texttt{Macaulay2} (see \cite{M2}) to compute the affine dimension of the ideal generated by the partial derivatives of $F$.
		Given that such dimension is~$0$, it follows that the Weddle locus $\{F=0\}$ is a smooth quartic.
	\end{example}
	
	\begin{remark}
		Coming back to the original Weddle problem (see \cite{Wed1850}), consider the~$4$-dimensional linear system of quadrics determined by~$6$ general points in $\PP^3$.
		By Theorem~\ref{prop_base_point_is_singular}, the associated Weddle locus has~$6$ singular points.
		There are examples of linear systems $\mathcal{L}$ of this type, such that these~$6$ points are the only singular points of $W(\mathcal{L})$ (for instance, see Example~\ref{ex_weddle_6_points} below).
		Hence, by Corollary~\ref{cor_sing_wed_less_10}, the rank of $\mathcal{L}$ is greater or equal to~$6$.
	\end{remark}
	
	\begin{example}\label{ex_weddle_6_points}
		Consider the linear system $\mathcal{L}$ determined by the~$6$ points
		\begin{align*}
			[1:0:0:0], [0:1:0:0], [0:0:1:0], [0:0:0:1], [1:1:1:1], [2:-1:5:-7]
		\end{align*}
		of $\PP^3$.
		We used \texttt{CoCalc} (see \cite{CoCalc}) to compute the affine dimension and the Hilbert polynomial of the ideal generated by the partial derivatives of a defining polynomial for $W(\mathcal{L})$.
		Since the affine dimension of this ideal is~$1$ and its Hilbert polynomial is~$6$, we conclude that $\operatorname{rk} \mathcal{L} \ge 6$.
	\end{example}

\section{Weddle loci of cyclic-symmetric tensors}\label{sect_weddle_semisym_tensors}
	In this last section, we go back to the case of linear systems of quadrics in $\PP(V)$ whose associated tensors are elements in ${\operatorname{N}_1}V$, \textit{i.e.}, cyclic-symmetric tensors.
	Since such linear systems are quite special, one may expect a special behavior for their Weddle loci.
	
	Set $V \cong \C^{n+1}$, and fix $\mathcal{L} \subset \PP({\operatorname{S}}^2 V)$ a linear system with $\dim \mathcal{L} = n+1$.
	Suppose that $\mathcal{L}$ is generated by the quadrics $Q_0,\dots,Q_n$.
	By a slight abuse of notation, for~$0 \le k \le n$, we identify the polynomial expression of the quadric $Q_k $ in the coordinates $(x_0,\dots,x_n)$ with its matrix representation by using the same symbol, namely $Q_k$.
	
	We prove the following fundamental property for the generators of linear systems associated to cyclic-symmetric tensors.
	\begin{lemma}\label{lem_relation_quadric_semisym_tensor}
		Let $\mathcal{L}$ be a linear system whose associated tensor $T$ is an element of ${\operatorname{N}_1}V$.
		Then, in the homogeneous coordinates $[x_0:\dots:x_n]$ of $\PP(V)$ the generators $Q_0,\dots,Q_n$ of $\mathcal{L}$ satisfy the relation
		\begin{equation*}
			x_0 Q_0 + \dots + x_n Q_n = 0 .
		\end{equation*}
	\end{lemma}
	\begin{proof}
		We proceed by induction on $n$.
		
		For $n=1$, we already mentioned in Example~\ref{ex_residue_tensor_dim2} that the matrices associated to $Q_0$ and $Q_1$ must have the form
		\[
		\begin{bmatrix}
			0 & a \\
			a & 2b
		\end{bmatrix}
		\text{ and }
		\begin{bmatrix}
			-2a & -b \\
			-b & 0
		\end{bmatrix}
		\]
		respectively.
		Therefore,
		\begin{align*}
			x_0 Q_0 &+ x_1 Q_1 \\
			&= x_0 (2a x_0 x_1 + 2b x_1^2) + x_1 (-2a x_0^2 - 2b x_0 x_1)  \\
			&=0
		\end{align*}
		and the equation is satisfied.
		
		Let $\mathcal{L}$ be associated to a tensor $T \in {\operatorname{N}_1}\C^{n+1}$, whose faces $T_0, \dots, T_n$ are the matrices associated to the generators $Q_0, \dots, Q_n$ of $\mathcal{L}$.
		For $0 \le k < n$, the matrix $T_k$ has the form
		\begin{equation*}
			\begin{bmatrix}
				& & & T_{0nk} \\
				& S_k& & \vdots \\
				& & & \vdots \\
				T_{0nk} & \dots & \dots & 2 T_{nnk}
			\end{bmatrix}
		\end{equation*}
		for some symmetric matrix $S_k$, whose associated quadric is denoted by $q_k$.
		Therefore, the equation
		\[
		q_k + \sum_{s = 0}^{n} T_{snk} x_s x_n = 0
		\]
		defines the quadric $Q_k$.
		
		By Proposition~\ref{thm_construction_semisym_tensors_inverse}, the matrices $S_0, \dots, S_{n-1}$ are the faces of a cyclic-symmetric tensor $S \in {\operatorname{N}_1}\C^n$.
		By the induction step, the relation
		\[
		x_0 q_0 + \dots + x_{n-1} q_{n-1} = 0
		\]
		holds.
		
		As already seen in the proof of Proposition~\ref{thm_construction_semisym_tensors} (specifically, due to the two symmetries that cyclic-symmetric tensors satisfy), the last face of $T$ is of the form
		\begin{equation*}
			\begin{bmatrix}
				- 2 T_{0n0} & - (T_{0 n 1} + T_{1 n 0}) & \dots & - T_{nn0} \\
				- (T_{0 n 1} + T_{1 n 0}) & - 2 T_{1n1} & \dots & - T_{nn1} \\
				\vdots & \vdots & \ddots & \vdots \\
				- T_{nn0} & - T_{nn1} & \dots & 0
			\end{bmatrix}
		\end{equation*}
		hence, the equation
		\[
		-\sum_{k=0}^{n-1} T_{knk}x_k^2 - \sum_{j=0}^{n-1} T_{nnj}x_j x_n - \sum_{0 \le \ell < m < n} (T_{\ell nm} + T_{mn\ell}) x_\ell x_m = 0
		\]
		defines the corresponding quadric $Q_n$.
		
		We compute, 
		\begin{align*}
			&\sum_{k=0}^{n-1} x_k Q_k \\
			&=  \sum_{k=0}^{n-1} x_k \left(q_k + \sum_{s = 0}^{n} T_{snk} x_s x_n\right) \\
			&=  \cancelto{\text{induction}}{\sum_{k=0}^{n-1} x_k q_k} + x_n \sum_{s = 0}^{n} x_s \sum_{k=0}^{n-1} T_{snk} x_k \\
			&=  x_n \left( \underbrace{\sum_{k=0}^{n-1} T_{knk}x_k^2}_{s=k<n} + \underbrace{\sum_{k=0}^{n-1} T_{nnk} x_k  x_n}_{s=n} + \underbrace{\sum_{0 \le s < k < n} (T_{s n k} + T_{kns}) x_s x_k}_{\text{otherwise}} \right) \\
			&= -x_n Q_n
		\end{align*}
		which proves the claimed equation.
	\end{proof}
	
	From now on, let $\mathcal{L}$ be a linear system whose associated tensor is a \emph{general} element of the linear space ${\operatorname{N}_1}V$.
	We want to study its Weddle locus $W(\mathcal{L})$.
	\begin{remark}\label{rem_base_point_semisym_dim_2}
		Suppose that $\dim V = 2$.
		By Example~\ref{ex_residue_tensor_dim2}, we have that the faces $Q_0$ and $Q_1$ of a cyclic-symmetric tensor $T \in {\operatorname{N}_1}\C^2$ give the equations
		\[
		a x_0 x_1 + b x_1^2 = 0 \ \ \text{and} \ \ a x_0^2 + b x_0 x_1 = 0
		\]
		respectively.
		Hence, $[-b,a]$ is a base point for the linear system associated to $T$.
	\end{remark}
	
	One can show (see also Corollary~\ref{cor_lin_sys_conics_semisym} below) that the linear system of conics associated to a general tensor in ${\operatorname{N}_1}\C^3$ has~$3$ base points.
	By Theorem~\ref{prop_base_point_is_singular}, this implies that its Weddle locus has~$3$ singular points.
	
	Then, Remark \ref{rem_base_point_semisym_dim_2} suggests that the Weddle locus of systems associated with cyclic-symmetric tensors is somehow special.
	In general, we prove that these Weddle loci always have singular points, and the number of the singular points is bounded from below.
	Before providing the full statement, we give the following.
	\begin{definition}
		A non-negative integer is called a \emph{Jacobsthal number} if it belongs to the following sequence, defined by the recurrence relation
		\[
		J_n = \begin{cases}
			0 &\text{if } n=0 \\
			1 &\text{if } n=1 \\
			J_{n-1} + 2J_{n-2} &\text{if } n>1
		\end{cases}
		\]
	\end{definition}
	We refer to \cite{Hor96}, and references therein, for results regarding the Jacobsthal sequence.
	In particular, we are interested in the recursion formula for the $(n+1)$-th Jacobsthal number
	\begin{equation}\label{prop_recurrence_Jacobsthal}
		J_{n+1} = 2^n - J_n ,
	\end{equation}
	which can be easily verified by induction through the computation
	\begin{align*}
		J_{n+1} &= J_{n} + 2J_{n-1}\\
		&= 2^{n-1} - J_{n-1} + 2(2^{n-2} - J_{n-2})\\
		&= 2^n - (J_{n-1} + 2J_{n-2}) .
	\end{align*}
	Furthermore, Jacobsthal numbers satisfy also the recursion formula
	\begin{equation*}
		J_{n+1} = 2 J_n + (-1)^n ,
	\end{equation*}
	and can be computed directly by means of the equation
	\[
	J_n = \frac{2^n - (-1)^n}{3}.
	\]
	See \cite{Hor96}, and references therein, for further details.

	The first Jacobsthal numbers are:
	\[
	0, 1, 1, 3, 5, 11, 21, 43, 85, 171, 341, 683, 1365, 2731,\dots
	\]
	and we immediately note that $J_2 = 1$ and $J_3 = 3$ are the number of base points of a linear system associated to a tensor in ${\operatorname{N}_1}\C^2$ and ${\operatorname{N}_1}\C^3$ respectively.
	
	We are now ready to prove the main result of this section.
	\begin{theorem}\label{thm_Jacobsthal_base_points}
		If $\dim V = n$, then a linear system $\mathcal{L}$ associated to a general cyclic-symmetric tensor in ${\operatorname{N}_1}V$ has $J_n$ base points, where $J_n$ is the $n$-th {Jacobsthal number}.
		In particular, by Theorem~\ref{prop_base_point_is_singular}, its Weddle locus $W(\mathcal{L})$ has at least $J_n$ singular points.
	\end{theorem}
	\begin{proof}
		We prove this theorem by induction.
		
		The cases $n = 0, 1$ are completely trivial, while the case $n = 2$ is the content of Remark~\ref{rem_base_point_semisym_dim_2}.
		
		For general $n$, we want to prove that a linear system $\mathcal{L}$ associated to a general cyclic-symmetric tensor $T \in {\operatorname{N}_1}\C^{n+1}$ has $J_{n+1}$ base points.
		As usual, with a slight abuse of notation, for $0 \le k \le n$, we use the same symbol, namely $Q_k$, to denote both the quadric determined by the $k$-th face $T_k$ of $T$ and the defining polynomial of that quadric.
		In the previous Lemma~\ref{lem_relation_quadric_semisym_tensor}, we proved the relation
		\begin{equation}\label{eq_relation_cyclic_proof}
			x_0 Q_0 + \dots + x_n Q_n = 0
		\end{equation}
		that the generators of $\mathcal{L}$ satisfy.
		Therefore, the variety $x_n Q_n$ contains the base locus $\mathcal{B}_n$ of $Q_0, \dots, Q_{n-1}$.
		
		By the proof of Lemma~\ref{lem_relation_quadric_semisym_tensor}, observe that, for $0 \le k < n$, the quadric
		\[
		q_k := \left. {Q_k}\right|_{ \{ x_n=0 \} }
		\]
		is a generator of a linear system in $\{ x_n=0 \} \cong \PP^{n-1}$ associated to a cyclic-symmetric tensor $S$ in ${\operatorname{N}_1}\C^n$, and $q_k$ corresponds to the face $S_k$ of $S$.
		The tensor $S$ is general as it is the upper-left-front sub-tensor of $T$, which is general by induction.
		This means that on the hyperplane $\{x_n=0\}$, the intersection
		\[
		q_0 \cap \dots \cap q_{n-1}
		\]
		consists, by inductive hypothesis, of $J_{n}$ distinct points.
		
		Since, for each $k = 0,\dots,n-1$, $T_k$ is obtained from $S_k$ by freely adding the elements of the last column, see Remark~\ref{rem_free} and the proof of Proposition~\ref{thm_construction_semisym_tensors}, we have that, by construction, $Q_k$ is the \emph{generic} quadric whose restriction to the hyperplane $\{ x_n=0 \}$ is $q_k$.
		Hence, the base locus $\mathcal{B}_n$ is a $0$-dimensional reduced scheme by Bertini's theorem.
		Therefore, by Bézout's theorem, $\mathcal{B}_n$ consists of~$2^n$ distinct points.
		Since, as already mentioned, the hyperplane $\{x_n=0\}$ contains $J_{n}$ of these~$2^n$ points, it follows by Equation~\eqref{eq_relation_cyclic_proof} that the quadric $Q_n$ contains the remaining $2^n - J_n$ points of $Q_0 \cap \dots \cap Q_{n-1}$.
		We conclude by applying the recurrence formula~\eqref{prop_recurrence_Jacobsthal} for Jacobsthal numbers.
	\end{proof}
	
	\begin{corollary}\label{cor_lin_sys_conics_semisym}
		Suppose that $\mathcal{L}$ is a linear system of conics in $\PP^2$ associated to a general cyclic-symmetric tensor in ${\operatorname{N}_1}\C^3$.
		Then, $W(\mathcal{L})$ splits in the union of~$3$ lines (\textit{cf.} Section~\ref{subsect_n=2}).
	\end{corollary}
	\begin{proof}
		By Theorem~\ref{thm_Jacobsthal_base_points}, $\mathcal{L}$ has $J_3 = 3$ base points.
		Therefore, by Theorem~\ref{prop_base_point_is_singular}, the Weddle locus of $\mathcal{L}$ has~$3$ singular points.
		Being $W(\mathcal{L})$ a plane cubic with~$3$ singular points, it splits as a union of~$3$ lines.
	\end{proof}
	
	\begin{remark}
		Suppose that $\mathcal{L}$ is a general linear system of quadrics in $\PP^3$ and $T_{\mathcal{L}}$ its associated tensor.
		Then it is interesting to note the following:
		\begin{itemize}
			\item if $\operatorname{rk} \mathcal{L} = 5$, {then $W(\mathcal{L})$ has~$10$ singular points} (see Theorem~\ref{thm_dim4_rk5});
			\item if $T_{\mathcal{L}} \in {\operatorname{N}_1}\C^4$, then $W(\mathcal{L})$ has~$5$ singular points (see Theorem~\ref{thm_Jacobsthal_base_points});
			\item if $\mathcal{L}$ comes from the original Weddle problem (see \cite{Wed1850}), then $W(\mathcal{L})$ has~$6$ singular points.
		\end{itemize}
		It is therefore natural to ask whether some of the above implications are reversible.
	\end{remark}

\bibliographystyle{alphaurl} 
\bibliography{bib_chiantini_fagioli}

\newcommand{\etalchar}[1]{$^{#1}$}
\begin{thebibliography}{BBCC13}

\bibitem[AB09]{AB09}
Hirotachi Abo and Maria~Chiara Brambilla.
\newblock Secant varieties of {S}egre-{V}eronese varieties {$\PP^n\times\PP^m$}
  embedded by {$\mathcal O(1,2)$}.
\newblock {\em Experiment. Math.}, 18(3):369--384, 2009.
\newblock URL: \url{http://projecteuclid.org/euclid.em/1259158472}.

\bibitem[BBCC13]{BBCC13}
Edoardo Ballico, Alessandra Bernardi, Maria~Virginia Catalisano, and Luca
  Chiantini.
\newblock Grassmann secants, identifiability, and linear systems of tensors.
\newblock {\em Linear Algebra Appl.}, 438(1):121--135, 2013.
\newblock \href {https://doi.org/10.1016/j.laa.2012.07.045}
  {\path{doi:10.1016/j.laa.2012.07.045}}.

\bibitem[BC19]{BC19}
Cristiano Bocci and Luca Chiantini.
\newblock {\em An introduction to algebraic statistics with tensors}, volume
  118 of {\em Unitext}.
\newblock Springer, Cham, 2019.
\newblock \href {https://doi.org/10.1007/978-3-030-24624-2}
  {\path{doi:10.1007/978-3-030-24624-2}}.

\bibitem[Bea00]{Bea00}
Arnaud Beauville.
\newblock Determinantal hypersurfaces.
\newblock {\em Michigan Math. J.}, 48:39--64, 2000.
\newblock \href {https://doi.org/10.1307/mmj/1030132707}
  {\path{doi:10.1307/mmj/1030132707}}.

\bibitem[BFP24]{BFP24}
Davide Bricalli, Filippo~Francesco Favale, and Gian~Pietro Pirola.
\newblock On the {H}essian of cubic hypersurfaces.
\newblock {\em Int. Math. Res. Not. IMRN}, 2024(10):8672--8694, 2024.
\newblock \href {https://doi.org/10.1093/imrn/rnad324}
  {\path{doi:10.1093/imrn/rnad324}}.

\bibitem[CC01]{CC01}
Luca Chiantini and Marc Coppens.
\newblock Grassmannians of secant varieties.
\newblock {\em Forum Math.}, 13(5):615--628, 2001.
\newblock \href {https://doi.org/10.1515/form.2001.025}
  {\path{doi:10.1515/form.2001.025}}.

\bibitem[CFF{\etalchar{+}}22]{CFF+22}
Luca Chiantini, {\L}ucja Farnik, Giuseppe Favacchio, Brian Harbourne, Juan
  Migliore, Tomasz Szemberg, and Justyna Szpond.
\newblock Configurations of points in projective space and their projections,
  2022.
\newblock \href {https://arxiv.org/abs/2209.04820} {\path{arXiv:2209.04820}}.

\bibitem[CG14]{CGer14}
Luca Chiantini and Anthony~V. Geramita.
\newblock On the determinantal representation of quaternary forms.
\newblock {\em Comm. Algebra}, 42(11):4948--4954, 2014.
\newblock \href {https://doi.org/10.1080/00927872.2013.828737}
  {\path{doi:10.1080/00927872.2013.828737}}.

\bibitem[CM21]{CM21}
Luca Chiantini and Juan Migliore.
\newblock Sets of points which project to complete intersections, and
  unexpected cones.
\newblock {\em Trans. Amer. Math. Soc.}, 374(4):2581--2607, 2021.
\newblock \href {https://doi.org/10.1090/tran/8290}
  {\path{doi:10.1090/tran/8290}}.

\bibitem[Emc25]{Emch25}
Arnold Emch.
\newblock On the {W}eddle surface and analogous loci.
\newblock {\em Trans. Amer. Math. Soc.}, 27(3):270--278, 1925.
\newblock \href {https://doi.org/10.2307/1989103} {\path{doi:10.2307/1989103}}.

\bibitem[Ful97]{Ful97}
William Fulton.
\newblock {\em Young tableaux}, volume~35 of {\em London Mathematical Society
  Student Texts}.
\newblock Cambridge University Press, Cambridge, 1997.
\newblock With applications to representation theory and geometry.

\bibitem[GKZ08]{GKZ94}
Israel~M. Gelfand, Mikhail~M. Kapranov, and Andrei~V. Zelevinsky.
\newblock {\em Discriminants, {R}esultants, and {M}ultidimensional
  {D}eterminants}.
\newblock Modern Birkh\"{a}user Classics. Birkh\"{a}user Boston, Inc., Boston,
  MA, 2008.
\newblock Reprint of the 1994 edition.

\bibitem[GS]{M2}
Daniel~R. Grayson and Michael~E. Stillman.
\newblock Macaulay2, a software system for research in algebraic geometry.
\newblock Available at \url{http://www2.macaulay2.com}.

\bibitem[Hor96]{Hor96}
Alwyn~Francis Horadam.
\newblock Jacobsthal representation numbers.
\newblock {\em Fibonacci Quart.}, 34(1):40--54, 1996.
\newblock \href {https://doi.org/10.1080/00150517.1996.12429096}
  {\path{doi:10.1080/00150517.1996.12429096}}.

\bibitem[IR22]{IR22}
Yakov Itin and Shulamit Reches.
\newblock Decomposition of third-order constitutive tensors.
\newblock {\em Math. Mech. Solids}, 27(2):222--249, 2022.
\newblock \href {https://doi.org/10.1177/10812865211016530}
  {\path{doi:10.1177/10812865211016530}}.

\bibitem[Lan12]{Lan12}
Joseph~M. Landsberg.
\newblock {\em Tensors: geometry and applications}, volume 128 of {\em Graduate
  Studies in Mathematics}.
\newblock American Mathematical Society, Providence, RI, 2012.
\newblock \href {https://doi.org/10.1090/gsm/128} {\path{doi:10.1090/gsm/128}}.

\bibitem[{Sag}20]{CoCalc}
{Sagemath, Inc.}
\newblock {CoCalc – Collaborative Calculation and Data Science}, 2020.
\newblock \url{https://cocalc.com}.

\bibitem[Wed50]{Wed1850}
Thomas Weddle.
\newblock On the theorems in space analogous to those of {P}ascal and
  {B}rianchon in a plane.---{P}art {II}.
\newblock {\em Cambridge and Dublin Mathematical Journal}, 5:58--69, 1850.

\end{thebibliography}

\end{document}